\documentclass[11pt,letterpaper]{amsart}
\usepackage{amsmath,amsthm,amsfonts,amscd,amssymb,latexsym,mathrsfs,graphicx,color,mathdots}

\usepackage{hyperref}

\theoremstyle{plain}
\newtheorem{theorem}{Theorem}
\newtheorem{corollary}[theorem]{Corollary}
\newtheorem{lemma}[theorem]{Lemma}
\newtheorem{proposition}[theorem]{Proposition}
\newtheorem{prop}[theorem]{Proposition}
\theoremstyle{definition}
\newtheorem{remark}[theorem]{Remark}
\newtheorem{example}[theorem]{Example}

\newtheorem{definition}[theorem]{Definition}
\newtheorem{conjecture}[theorem]{Conjecture}
\numberwithin{theorem}{section}

\newcommand{\BR}{\mathbf{R}}
\newcommand{\BZ}{\mathbf{Z}}

\newcommand{\BC}{\mathbf{C}}

\newcommand{\BQ}{\mathbf{Q}}

\newcommand{\BA}{\mathbf{A}}

\newcommand{\ten}{\otimes}
\newcommand{\Res}{\text{Res}}

\renewcommand{\ten}{\otimes}

\newcommand{\fu}{\mathfrak{u}}
\newcommand{\ip}[2]{\langle #1, #2 \rangle}
\newcommand{\ad}{\text{ad}}

\newcommand{\Hom}{\text{Hom}}

\newcommand{\MO}{\mathcal{O}}

\newcommand{\mG}{\mathcal{G}}
\newcommand{\mH}{\mathcal{H}}
\newcommand{\mI}{\mathcal{I}}
\newcommand{\mS}{\mathcal{S}}
\newcommand{\mO}{\mathcal{O}}
\newcommand{\mE}{\mathcal{E}}
\newcommand{\mC}{\mathcal{C}}
\newcommand{\mL}{\mathcal{L}}
\newcommand{\fa}{\mathfrak{a}}

\newcommand{\fp}{\mathfrak{p}}
\newcommand{\fq}{\mathfrak{q}}
\newcommand{\fg}{\mathfrak{g}}
\newcommand{\fl}{\mathfrak{l}}

\newcommand{\fn}{\mathfrak{n}}
\newcommand{\fk}{\mathfrak{k}}
\newcommand{\ft}{\mathfrak{t}}
\newcommand{\fs}{\mathfrak{s}}

\newcommand{\fX}{\mathfrak{X}}

\newcommand{\tr}{\operatorname{tr}}

\newcommand{\id}{\text{id}}

\newcommand{\del}{\delta}
\newcommand{\Del}{\Delta}

\renewcommand{\phi}{\varphi}

\newcommand{\dom}{\backslash}
\newcommand{\tbt}[4]{\left(\begin{array}{cc} #1&#2\\#3&#4 \end{array}\right)}

\newcommand{\Vol}{\text{Vol}}

\newcommand{\BAF}{\BA_F}
\newcommand{\BAFf}{\BAF^f}
\newcommand{\BAE}{\BA_E}

\renewcommand{\det}{\mathrm{det}}
\newcommand{\norm}{{\rm Nm}}
\newcommand{\GalEF}{\Gamma_{E/F}}

\newcommand{\Aut}{L^2_{\mathrm{disc}}}

\newcommand{\Idisc}{I_{\mathrm{disc}}}
\newcommand{\Rdisc}{R_{\mathrm{disc}}}
\newcommand{\Rdisci}{R_{{\mathrm{disc}}, \psi}}

\newcommand{\Idisci}{I_{{\mathrm{disc}}, \psi}}
\newcommand{\Idiscintro}{I_{{\mathrm{disc}}, \psi}}
\newcommand{\Idisctc}{I_{\mathrm{disc},t,c}}

\newcommand{\SHdisci}{{S}^H_{\text{disc},\psi}}
\newcommand{\SHdiscintro}{{S}^H_{\text{disc},\psi}}

\newcommand{\SHpsidisci}{{S}^{H_\psi}_{\text{disc},\psi}}

\newcommand{\Heckev}{\mH(G_v)}

\newcommand{\GN}{G(N)}
\newcommand{\Psiqs}{\Psi(G^*, \emb)}

\newcommand{\emb}{\eta_{\kappa}}
\newcommand{\embend}{\xi_{\underline{\kappa}}}
\newcommand{\bcend}{\eta_{\kappa \cdot \underline{\kappa}}}
\newcommand{\PsiG}{\Psi(G,\xi)}
\newcommand{\disc}{\mathrm{\disc}}

\newcommand{\pa}[2]{\langle #1, #2 \rangle} 
\newcommand{\Spsi}{\mathcal{S}_\psi} 
\newcommand{\G}{G} 
 
\newcommand{\U}{G} 
 
\newcommand{\cp}{\fp}

\author[M. Gerbelli-Gauthier]{Mathilde Gerbelli-Gauthier}
\email{mathilde.gerbelli-gauthier@mcgill.ca}
\address{Department of Mathematics and Statistics, McGill University}

\title[Limit Multiplicity and The Stable Trace Formula]
{Limit Multiplicity For Unitary Groups and The Stable Trace Formula}

\begin{document}

\begin{abstract}
We give upper bounds on limit multiplicities of certain non-tempered representations of unitary groups $U(a,b)$,  conditionally on the endoscopic classification of representations. Our result applies to some cohomological representations, and we give applications to the growth of cohomology of cocompact arithmetic subgroups of unitary groups. The representations considered are transfers of products of characters and discrete series on endoscopic groups, and the bounds are obtained using Arthur's stabilization of the trace formula and the classification established by Mok and Kaletha-Minguez-Shin-White.
\end{abstract}

\maketitle 

\tableofcontents

\section{Introduction.}

Let $G$ be a semisimple Lie group and let $\Gamma \subset G$ be an arithmetic lattice. Such a group is an analogue of the ``$\BZ$-points of $G$": is realized as the intersection $\mG(\BQ) \cap K \subset \mG(\BA_f)$ for a choice of algebraic group $\mG/\BQ$ such that $\mG(\BR) = G$ and $K$ a compact-open subgroups of the finite adelic points of $G$. A natural invariant to study is the group cohomology $H^i(\Gamma,\BC)$. Yet beyond some low-rank examples, the dimension of this cohomology has only been computed explicitly for specific instances of $\Gamma$, for example \cite{CT20, GMY21}. A variant of the problem is to study this question \emph{in towers}: one studies the asymptotic properties of $\dim H^i(\Gamma_n, \BC)$ for sequences $\Gamma_n$ of nested subgroups as $n \to \infty$, see for example \cite{ CE, CM13, MaSh18}.

A central family of such sequences $\Gamma_n$ are congruence towers $\Gamma(p^n)$. These are obtained by fixing a suitable prime $p$ and considering sequences of subgroups $K(p^n) = K^pK_p(p^n)$.  The group $K^p$ is a fixed compact-open subgroup of $G(\BA_f^p)$, the finite adelic points away from $p$, and \[K_p(p^n) = \mG(\BQ_p) \cap \{g \in GL_n(\BZ_p) \mid g \equiv I \mod p_n\}\] for a choice of embedding $\mG(\BQ_p) \hookrightarrow GL_n(\BQ_p)$. The resulting nested sequence of subgroups $\Gamma(p^n) = \mG(\BQ) \cap K(p^n)$ are referred to as principal $p$-power congruence towers.

This article is motivated by the study of rates of growth of~$\dim H^i(\Gamma(p^n), \BC)$ as~$n$ grows, for $\Gamma(p^n)$ cocompact. 

These dimensions can be expressed representation-theoretically by using Matsushima's formula \cite{Ma67}:
 \[ H^i(\Gamma(p^n), \BC) = \bigoplus_{\pi}m(\pi, p^n)H^i(\fg,K;\pi). \] Here the sum is taken over isomorphism classes of unitary representations of~$G$, the number~$m(\pi, p^n)$ is the multiplicity of~$\pi$ in the regular representation of~$G$ on~$L^2(\Gamma(p^n) \dom G)$, and $H^i(\fg,K;\pi)$ is the $(\fg,K)$-cohomology of $\pi$. Following the work of Vogan-Zuckerman \cite{VZ}, the finitely many  representations contributing non-trivially to the above sum are well-understood. Thus the question is reduced to the growth of multiplicities~$m(\pi, p^n)$ of cohomological representations.

Multiplicity growth rates are best understood for discrete series representations, which contribute to cohomology only in the middle degree. In that case, DeGeorge-Wallach \cite{DGW} and later Savin \cite{Sa89} have shown that $m(\pi, p^n)$ grows proportionally to the index $[\Gamma(1): \Gamma(p^n)]$. This leaves open the question of multiplicity growth for cohomological representations in lower degrees. In general, these are non-tempered, and DeGeorge-Wallach show that their multiplicities $m(\pi, p^n)$ satisfy \[m(\pi, p^n)/[\Gamma(1):\Gamma(p^n)] \xrightarrow[n \to \infty]{} 0.\] 
Sarnak-Xue \cite{SX91} have predicted upper bounds on growth, interpolating between the rate for discrete series and the constant multiplicity of the trivial representation. Here ``$f(n) \ll g(n)$" means that for $n$ large enough, $f(n)$ is bounded by a constant multiple of $g(n)$, and $\ll_{\epsilon}$ indicates that the implied constants depends on $\epsilon$.
\begin{conjecture}(Sarnak-Xue) \label{conj SX} Let $\pi$ be a unitary representation of $G$ and let \[p(\pi) = \inf\{p \geq 2\,|\, \text{ the $K$-finite matrix coefficients of $\pi$ are in $L^p(G)$}\}.\] Then \[ m(\pi,p^n) \ll_{\epsilon} [\Gamma(1): \Gamma(p^n)]^{\frac{2}{p(\pi)}+\epsilon}. \] 
\end{conjecture}
 By definition, the representation~$\pi$ is tempered if~$p(\pi) =2$. Thus Sarnak-Xue expect the failure of temperedness to dictate the rate of growth of~$m(\pi, p^n)$. 
\subsection{Main Theorem}
In this article, we give upper bounds on the multiplicity growth of certain cohomological representations of unitary groups. The results are conditional on the endoscopic classification of representations, as discussed in \ref{sectiononconditionality}.  Let~$E/F$ be a CM extension of number fields with~$F \neq \BQ$,  and~$\cp$ a prime of~$F$ such that the cardinality $\norm(\cp)$ of the residue field is large enough, see Section \ref{Choice of Test Functions}. Let~$a\leq \frac{N}{2}$ and let ~$G$ be a unitary group defined from a Hermitian form of signature $((a,N-a),(N,0),...,(N,0))$ relative to~$E/F$. Finally, let~$\Gamma(\cp^n)$ be a sequence of principal level cocompact lattices in~$G$, defined in Section \ref{section on limit multiplicity}. Denote by $\nu(n)$ the unique $n$-dimensional irreducible representation of $SL_2(\BC)$. Our main theorem concerns cohomological representations $\pi$ of $G$ satisfying the following two conditions: 

\begin{itemize}
	\item[(i)] $\pi$ belongs to a cohomological Arthur packet associated to a parameter $\psi$ with $\psi \mid_{SL_2(\BC)} = \nu(2k) \oplus \nu(1)^{N-2k}$;
	\item[(ii)] $\pi$ does not appear in any other Arthur packet.
\end{itemize}
Such $\pi$ are endoscopic transfers of products of discrete series with characters, from endoscopic groups of $G$ of the form $U(2k) \times U(N-2k)$. They exist if $a+k\geq N/2$. For example, they include a family $\pi_k$ described in Section \ref{section on cohomology} and which contributes to cohomology in degrees \[i = i(N,a,k) = \begin{cases}
\left(\frac{N-2}{2}\right)^2 + 2a -k^2 & \text{$N$ even} \\ \left(\frac{N-1}{2}\right)^2 + a -k^2 & \text{$N$ odd}.
\end{cases}\]

We recall that $f(n) \ll g(n)$ means that for $n$ large enough, $f$ is bounded by a constant multiple of $g$. 

 \begin{theorem}\label{main theorem intro} Assume the endoscopic classification of representations for unitary groups stated in \cite{KMSW}. Let $\Gamma(\cp^n)$ be a tower of principal level cocompact lattices in $G = U(a,N-a)$, such that the size $\norm(\cp)$ of the residue field is large enough. Let $\frac{N}{2} > k\geq \frac{N}{2}-a$, and let $\pi$ be a cohomological representation of $G$ satisfying properties (i) and (ii) above. Then \[ m(\pi, \cp^n) \ll \norm(\cp^n)^{N(N-2k)}.  \] In particular, Conjecture \ref{conj SX} holds for $\pi$. 
\end{theorem}

Our method of proof leads us to believe that these bounds are sharp, in the sense that one should be able to achieve them for a suitable choice of lattices. Indeed, our strategy is to decompose the multiplicity count and show that the leading term comes from a smaller group for which exact asymptotics are known. We expect that the other terms can be made to oscillate and not contribute in the limit. For $G = U(2,1)$, this type of method was carried out successfully by Simon Marshall \cite{Ma14}. 

Our representations do not account for all of the cohomology, but in some low degrees, we expect them to do so asymptotically. For example, the smallest nonzero degree~$i$ for which~$H^i(\Gamma(\cp^n), \BC)$ is nontrivial is~$i=a$. When~$N$ is odd, the representations associated to $k = \frac{N-1}{2}$ contribute asymptotically all the cohomology in degree~$a$, yielding the following bounds. 

\begin{corollary}
	Under the assumptions of Theorem \ref{main theorem intro}, assume additionally that $N$ is odd. Then \[ \dim H^a(\Gamma(\cp^n), \BC) \ll \norm(\cp^n)^{N}. \]
\end{corollary}

In order to describe the more general range of degrees in which we predict that the representations we can control contribute all the cohomology, we state our main technical theorem.  It concerns bounds on limit multiplicity for representations belonging to a prescribed archimedean Arthur packet. This result does not require that the representations be cohomological, and our most general limit multiplicity result is the following. 

\begin{theorem} 
Under the assumptions of Theorem \ref{main theorem intro}, let $\psi_{\infty}$ be an Arthur parameter with regular infinitesimal character and such that $\psi_{\infty}\mid_{SL_2(\BC)} = \nu(2k) \oplus \nu(1)^{N-2k}$. Let $\pi_\infty \in \Pi_{\psi_\infty}$, and let $\Psi_{\psi_{\infty}, 1}$ be the set of Arthur parameters for $G$ whose specialization at infinity is $\psi_{\infty}$, and associated to representations with trivial central character. Then \begin{equation}  \sum_{\psi \in \Psi_{\psi_{\infty}, 1}} \sum_{\pi = \pi_{\infty}\ten \pi_f   \in \Pi_{\psi}} m(\pi) \dim \pi_f^{K(\cp^n)} \ll \norm(\cp^n)^{N(N-2k)},  \end{equation}
	where $m(\pi)$ denotes the multiplicity of $\pi$ in $\Aut(G(F)\dom \G(\BAF), 1)$. 
\end{theorem}

These types of Arthur parameters seem to control the growth of certain degrees of cohomology. The combinatorics of intersections between various Arthur packets rapidly get complicated, but here is a sample of behavior we expect:
\begin{conjecture}
	Let $G=U(N-a,a)$ be as above. Then 
	\begin{itemize}
		\item[(i)] the representations belonging to Arthur packets attached to parameters $\psi$ with $\psi \mid_{SL_2} = \nu(N-\ell)\oplus \nu(1)^{\ell}$ contribute asymptotically all the cohomology in degrees $a \cdot \ell $ for $0\leq \ell \leq N-2(a-1)$, 
		\item[(ii)] for these degrees, \[\dim H^{a\cdot\ell}(\Gamma(\cp^n), \BC) \ll \norm(\cp^n)^{N\cdot \ell}. \]
	\end{itemize}

\end{conjecture}
The range of degrees to which the conjecture applies is larger for smaller values of~$a$, i.e. when $G$ is farther from being quasisplit. For $a=1$, Marshall-Shin \cite{MaSh18} proved (ii) under some assumptions on $\cp$, and conjectured (i).

\subsubsection{Outline of the Proof} The results are proved in the framework of endoscopy, Arthur parameters, and the stable trace formula. 
The theorem is a consequence of the endoscopic classification of representations for unitary groups. The classification is a result of Mok \cite{Mok} if the group $G$ is quasisplit, and of Kaletha--Minguez--Shin--White \cite{KMSW} for inner forms, building on the seminal work of Arthur \cite{A}.  It gives a decomposition of the regular representation of~$G(\BAF)$ on the discrete spectrum: \[\Aut(G(F) \dom G(\BAF)) \simeq \bigoplus_\psi \bigoplus_{\pi \in \Pi_{\psi}} m(\pi) \pi\] where the irreducible summands~$\pi = \ten_v' \pi_v$ are automorphic representations; they appear in the discrete spectrum with multiplicity~$m(\pi)$. This decomposition is given in terms of Arthur packets~$\Pi_{\psi}$ indexed by Arthur parameters~$\psi$. These parameters are formal objects \[\psi = \boxplus_i (\mu_i \boxtimes \nu(m_i)) \] where each $\mu_i$ is a cuspidal automorphic representation of $GL_{n_i}$ and $\nu(m_i)$ is the unique irreducible $m_i$-dimensional representation of $SL_2(\BC)$. Such a~$\psi$ is associated to a packet of representations of a unitary group of rank~$N$ if~$\sum_i n_im_i = N$ and if~$\psi$ is self-dual in a suitable sense. The parameters stand in for homomorphisms \[ \psi: L_F \times SL_2(\BC) \to {^L}G \] where~${^L}G$ is the~$L$-group of~$G$ and~$L_F$ is the Langlands group of~$F$, an object whose existence is at the present moment only hypothetical. Despite Arthur parameters being purely formal objects, one can consider the restriction~$\psi\mid_{SL_2(\BC)} := \oplus_i \nu(m_i)^{n_i}$ which is an actual finite-dimensional representation. The classification of parameters in terms of this restriction plays a central role in our argument, and we refer to the group $SL_2(\BC)$ used to build Arthur parameters as the ``Arthur $SL_2$". 

Endoscopy is a specific instance of the principle of functoriality in the Langlands program. It concerns certain groups~$H$, the so-called endoscopic groups of $G$, and states that if~$\psi$ factors through an embedding~${^L}H \hookrightarrow {^L}G$, then there must be trace identities between the characters of the representations~$\pi \in \Pi_{\psi}$ and those of representations~$\pi^H$ of~$H$ in a corresponding packet~$\Pi_\psi^H$. 
The character identities are witnessed through the trace formula~$\Idiscintro(f)$. In the case of our parameters with regular infinitesimal character, this distribution computes the trace of convolution by a smooth, compactly supported function~$f$ on the subspace of~$\Aut$ spanned by the representations $\pi \in \Pi_{\psi}$. More specifically, the character identities appear in a decomposition of $\Idiscintro(f)$ referred to as the stabilization of the trace formula (written here in a simplified version for exposition purposes): \begin{equation}\label{Stable Trace Formula Intro} \Idiscintro(f) = \sum_{H} \SHdiscintro(f^H).\end{equation} Here the sum runs over all endoscopic groups $H$ such that $\psi$ factors through ${^L}H$. The distributions $\SHdiscintro(f^H)$ are stable, meaning that they satisfy a strengthening of the conjugacy-invariance property of characters of representations. 

The summands $\SHdiscintro(f)$, initially defined inductively, can be expanded explicitly as linear combinations of traces $\tr \pi (f)$ of the representations $\pi \in  \Pi_{\psi}$: this is the so-called stable multiplicity formula. We write here a simplified version of the stable multiplicity formula in which we have omitted constants which can be ignored in the asymptotic questions we are concerned with: \begin{equation} \label{STFintro} S^H_{\text{disc}, \psi}(f^H) = \sum_{\pi \in \Pi_{\psi}} \xi(\pi, H)\tr \pi (f). \end{equation} The coefficients $\xi(\pi, H)$ arise from characters of a $2$-group $\mS_{\psi}$, the group of connected components of the centralizer of the image of $\psi$. More precisely, there are two mappings \begin{align*}
	\{\text{representations } \pi \in \Pi_{\psi}\} &\to \{ \text{characters of $\mS_{\psi}$} \} \\ 
	\{ \text{$H$ such that $\psi$ factors through }{^L}H \} & \to \{\text{elements of }\mS_{\psi} \},
\end{align*}
the second of which is a bijection. In this way, the coefficient $\xi(\pi,H)$ in the decomposition of the stable term $\SHdiscintro(f^H)$ is the value of the character associated to $\pi$ on the group element corresponding  to $H$.\\

In this context, the steps of the proof of Theorem \ref{main theorem intro} can be outlined as: 
\begin{enumerate}
	\item[(i)] (\S \ref{Arthur Packets of Cohomological Representations}) Determine the parameters $\psi$ associated to the packets containing cohomological representations. This relies on work of Arthur~\cite{AUC} and Adams-Johnson~\cite{AJ}.  
	\item[(ii)] (\S \ref{Choice of Test Functions}) Write the dimension of cohomology as $\sum_{\psi} \Idiscintro(f(\cp^n))$ for a specific function $f(\cp^n)$, summing over the parameters $\psi$ computed in the first step.  
	\item[(iii)] (\S \ref{Stabilization}) Fix a cohomological parameter $\psi$.  Use the stabilization of the trace formula to decompose \[\Idiscintro(f(\cp^n)) = \sum_H \SHdiscintro(f(\cp^n)^H). \] 
	\item[(iv)] (\S \ref{Bounds by the Dominang Group}) By interpreting the coefficients $\xi(\pi, H)$ appearing in the stable multiplicity formula~\eqref{STFintro} as values of characters of~$\mS_{\psi}$, conclude that there is a specific endoscopic group~$H_\psi$ whose contribution bounds that of all the others in \eqref{Stable Trace Formula Intro}, i.e. such that \[ \Idisci(f(\cp^n)) \leq K(\psi) \SHpsidisci(f(\cp^n)^{H_\psi}) \] for a constant $K(\psi)$ computed in terms $\psi\mid_{SL_2(\BC)}$ and of the number of irreducible summands of $\psi$, and which can be uniformly bounded in terms of the rank $N$ of the unitary group. The group $H_\psi$ depends only on $\psi\mid_{SL_2(\BC)}$. As such it is determined by the parameter $\psi_{\infty}$ and ultimately by the choice of cohomological representations. 
	\item[(v)] (\S \ref{section on hyperendoscopy} and \S \ref{section on our limit multiplicity results}) Bound the stable trace $\SHpsidisci(f(\cp^n)^{H_\psi})$ in terms of the multiplicity $m(\pi^{H_\psi}, \cp^n)$ for a family $\pi^{H_\psi}$ of representations of $H_\psi$. This relies on the fundamental lemma, proved by Laumon-Ng\^o for unitary groups~\cite{LN}, but also on a variant for congruence subgroups due to Ferrari \cite{Fe}. In order to control the discrepancy between~$S^{H_\psi}_{\text{disc}, \psi}$ and~$I^{H_\psi}_{\text{disc}, \psi}$, we make use of the notion of hyperendoscopy, also introduced by Ferrari. 
	\item[(vi)] (\S \ref{section on others' limit multiplicity results} and \S \ref{last section on limit multiplicity}) The representations~$\pi^{H_\psi}$ obtained via steps (i)-(v) from parameters such that $\psi \mid_{SL_2(\BC)} = \nu(2k) \oplus \nu(1)^{N-2k}$ are the product of a discrete series representation and a character. Their limit multiplicity is thus known by results of Savin~\cite{Sa89}, which gives the desired bounds. 
\end{enumerate}

\begin{remark}
	Some comments on possible extensions of the result: the proof exploits the fact that for a global $A$-parameter $\psi$, the restriction $\psi\mid_{SL_2(\BC)}$ is determined locally at any place. Here, archimedean restrictions associated to cohomological representations propagate to global and everywhere-local restrictions and induce slow rates of growth. But there is nothing special about infinity: similar methods could provide information about automorphic representations which belong to Arthur packets with large Arthur $SL_2$.

	The endoscopic classification was of course proved by Arthur~\cite{A} for quasisplit orthogonal and symplectic groups, and Ta\"ibi~\cite{T} extended the key result used here, namely the stable multiplicity formula, to some classes of inner forms. It is likely the case that similar methods could be a good starting point to provide analoguous bounds for these groups. 
	
	The restrictions on the types of representations we deal with are rooted in restrictions on the Arthur parameters we consider. These have two simple pieces which witness opposite extreme behaviors when restricted to the Arthur $SL_2$. This allows us to obtain bounds by ``applying endoscopy once". To extend the results to e.g. representations associated to global parameters with an arbitrary number of simple pieces, one could iterate the inductive process of steps (iv) and (v) and bound representations coming from hyperendoscopic groups, i.e. endoscopic groups of endoscopic groups. 
\end{remark}

Our proof method is in the lineage of a body of recent work applying the framework of endoscopy to growth of cohomology. Most notably, bounds on multiplicity growth of all non-tempered cohomological representations were obtained by Marshall~\cite{Ma14} for~$G = U(2,1)$, and Marshall-Shin~\cite{MaSh18} for~$G = U(N,1)$ and a level~$\fn$ divisible by primes splitting in the CM extension used to define the unitary group.

\subsection{Conditionality} \label{sectiononconditionality}

Our results are conditional on the endoscopic classification of representations for inner forms of unitary groups, a result which remains to be fully proved in several ways. As explained in the introduction of~\cite{KMSW} and in~\cite[\S 2.6]{Mok} the classification depends on upcoming work of Chaudouard-Laumon on the weighted fundamental lemma. It also depends, through its dependency on \cite{A}, on several papers of Arthur not yet made public. Moreover, the proof of the classification in~\cite{KMSW} is not itself complete: in particular, the results appearing here as Theorem~\ref{local character identities for general unitary groups} and Theorem~\ref{Rdisc computes the traces on the packet.} are only proved for generic parameters. A full proof is expected in~\cite{KMS}. 

\section*{Acknowledgements} 
This work was initially the author's doctoral thesis. She is grateful to her advisor Matt Emerton for many years of conversations, ideas, insights, and support. She also thanks James Arthur, Nicolas Bergeron, Laurent Clozel, Rahul Dalal, Shai Evra, Tasho Kaletha, Colette M{\oe}glin, Sarah Peluse, Peter Sarnak, and Joel Specter for 
helpful conversations, correspondence, and advice. A special thanks to Simon Marshall for an attentive reading that picked up some issues in an earlier version. She also thanks the anonymous referees for useful feedback which led to many improvements and clarifications. Finally, the author is grateful for support from the Natural Sciences and Engineering Research Council of Canada, and was supported by the Charles Simonyi Endowment at the Institute for Advanced Study.

\section{L-groups,  Parameters, and the Trace Formula }

\subsection{Notation} \label{section on notation}Let~$E/F$ be a CM extension of number fields with Galois group~$\Gamma_{E/F}$, algebraic closure~$\bar{F}$ and absolute Galois groups~$\Gamma_F$ and~$\Gamma_E$. We denote places of $F$ and~$E$ by~$v$ and~$w$ respectively, and let~$E_v = E \ten_F F_v$ for~$v$ a place of~$F$. Let~$F_\infty = F \ten_\BQ \BR$, the product of all archimedean completions of~$F$. Let~$\MO_F$ and~$\MO_E$ be rings of integers, and~$\BAF$ and~$\BAE$ be ad\`ele rings, with ~$\norm: \BAE \to \BAF$
the norm map. Let~$\BAFf$ be the finite ad\`eles, so that we have~$\BAF = F_\infty \times \BAFf$.

Fix $\chi_{\kappa}$ for $\kappa \in \{\pm 1\}$, a pair of Hecke characters of~$E$. We fix~$\chi_{+1}$ to be trivial and the character~$\chi_{-1}$ is chosen so that its restriction to~$\BAF/F^\times$ is the quadratic character associated to~$E$ by class field theory.  

If $F$ is a field and $G/F$ is a reductive group, we will denote the center of~$G$ by~$Z_G$ or by~$Z(G)$.  If $F$ is global, we denote~$G(F_v)$ by~$G_v$ and~$G(F_\infty)$ by~$G_\infty$. For~$H \subset G(\BAF)$, we use the notation~$H_f = H \cap G(\BAFf)$.  The complexified Lie algebra of~$G_\infty$ will be denoted~$\fg_\infty$.

\subsection{Unitary Groups and Their $L$-Groups} 

\subsubsection{Quasisplit Unitary Groups} \label{section on quasisplit unitary groups} We now introduce unitary groups and their~$L$-groups, following the exposition of Kaletha-Minguez-Shin-White \cite[\S 0]{KMSW}. Let~$E/F$ be a quadratic algebra: either the CM extension introduced above or one of its localizations~$E_v/F_v$, in which case we have~$E_v \simeq F_v \times F_v$ when~$v$ is split. If this is the case, fix an identification~$E_v = F_v \times F_v$. Let~$\sigma \in \text{Aut}_F(E)$ be the nontrivial element of~$\Gamma_{E/F}$ if~$E$ is a field, and the involution~$\sigma (x,y) = (y,x)$ if~$E = F \times F$. If $E$ is a split quadratic algebra, set $\Gamma_E := \Gamma_F$. Let~$\Phi_N$ be the antidiagonal~$N \times N$ matrix  \begin{equation} \label{Matrix Phi_N} \Phi_N =  \left(\begin{array}{ccccc}
		&& 1\\  &\iddots& \\  (-1)^{N-1} && 
	\end{array}\right).\end{equation}
 Let~$U_{E/F}(N)$ (sometimes denoted~$U(N)$) be the reductive group over ~$F$ with~$U_{E/F}(\bar{F}) \simeq GL_N(\bar{F})$, with Galois action \[ \tau_N(g) = \begin{cases}
\tau(g) & \tau \in \Gamma_E  \\ 
\text{Ad}(\Phi_N)\tau(g)^{-t} & \tau \in \Gamma_F \setminus \Gamma_E
\end{cases}. \] We have~$U_{E/F}(N,E) = GL_N(E)$, and we can identify \begin{equation} \label{definition of the unitary group in coordinates} U_{E/F}(N,F) = \{ g \in GL_N(E) \mid \text{Ad}(\Phi_N)\sigma(g)^{-t} = g \},  \end{equation}
a quasisplit unitary group with maximal (non-split) torus given by the group of diagonal matrices, and a Borel subgroup consisting of upper-triangular matrices. If~$E = F \times F$, we have~$U(N)  \simeq GL_N$ and we fix an isomorphism to identify them.

If~$F$ is global, we consider the various localizations of~$U(N,F)$. If~$v$ splits in~$E$, we have~$U(N,F_v) \simeq GL_N(F_v)$. Otherwise $U(N,F_v)$ is a quasisplit unitary group over~$F_v$, which determines it uniquely up to isomorphism as we shall see below. 

\subsubsection{Inner Forms} \label{section on inner forms}
An inner form of~$U(N)$ is a pair~$(G, \xi)$ consisting of an algebraic group~$G/F$ together with an isomorphism~$\xi: G(\bar{F}) \to U(N, \bar{F})$ such that for all~$\sigma \in \Gamma_F$, the automorphism~$\xi^{-1} \circ \sigma \circ \xi \circ  \sigma^{-1}$ is inner. Though the choice of~$\xi$ is always present, it will sometimes be implicit in our notation. We will denote~$U(N)$ by~$G^*$ when we want to highlight that it is the quasisplit form of~$G$.  
\begin{remark}
	In this article, we always require that the inner forms be groups defined with respect to a Hermitian space over~$E$.
\end{remark}

We now discuss which possible groups~$G$ can arise as inner forms of~$U_{E/F}(N)$ in the cases where~$F$ is local or global.
\subsubsection{Local Inner Forms and the Kottwitz Sign} \label{Kottwitz signs subsection}

If~$v$ is archimedean, the classification of inner forms is well-known: a unitary group over~$F_v = \BR$ is determined by its signature $p+q = N$, with~$U(p,q) \simeq U(q,p)$. Since the notation $U(N)$ is reserved for quasisplit groups, we denote the compact inner form of $U(N,\BR)$ by $U_N(\BR)$. 

For $v$ nonarchimedean, the classification of unitary groups coming from Hermitian forms over quadratic algebras over $F_v$ is due to Landherr~\cite{Lan36}: if~$N$ is odd, there is one class of Hermitian forms up to isomorphism, so the group~$U(N,F_v)$ is the unique unitary group of rank~$N$. If~$N$ is even, there are two isomorphism classes of unitary groups, only one of which (the one containing~$U(N,F_v))$ is quasisplit.

One associates to an inner form $G_v$ of  $U_{E_v/F_v}(N)$ a Kottwitz sign $e(G_v)$. We record the formulas for $e(G_v)$ as computed in \cite{Ko83}.

\begin{itemize}
	\item[-] For $F_v = \BR$, let $q(G_v)$ be half the dimension of the symmetric space associated to the group $G_v$. Then $e(G_v) = (-1)^{q(G_v)-q(G^*_v)}$.  
	\item[-] For $F_v$ non-archimedean, let $r(G_v)$ be the rank of $G_v$. Then $e(G_v)=(-1)^{r(G_v)-r(G_v^*)}$. 
\end{itemize}
Kottwitz proves in \cite{Ko83} that for $G$ defined over a global field, the local signs cancel out and $\prod_v e(G_v) = 1$. 

\subsubsection{Global Inner Forms}
We describe the classification of global forms of unitary groups, following the discussion in Section 0.3.3 of \cite{KMSW}. For $N$ odd, any collection of local inner twists, quasisplit at all but finitely many places, can be realized as the localization of a global inner twist. 

Fo $N$ even, the behavior of the place $v$ in $E$ determines cohomological invariants attached to $G_v$. For each $v$, we have $H^1(\Gamma_{F_v}, G^{*,ad}_v) \simeq \BZ/2\BZ$. If $v$ is split in $E$, the invariant of $G_v$ depends on the division algebra~$D_v$ such that~$G_v = \Res^{D_v}_{F_v}GL_{M_v}$. Since we only consider unitary groups coming from Hermitian forms, this invariant is always~$0$ for us. At finite nonsplit places, the quasisplit group~$U(N)_v$ and its unique inner form correspond respectively to~$0$ and~$1$ in~$\BZ/2\BZ$. At the infinite places, the signature~$(p,q)$ determines the invariant~$\frac{N}{2}+q \in \BZ/2\BZ$. For a collection of local~$G_v$ to come from a global unitary group, the all but finitely many nonzero invariants associated to~$G_v$ must sum up to zero. Consequently, we have: 

\begin{lemma} \label{classification of global inner forms}
	Let~$E/F$ be a CM extension of number fields. There exists an inner form $G$ of~$U_{E/F}(N)$ with any choice of signature at the infinite places. Moreover~$G$ can be chosen to be quasisplit outside of a set of places of size at most~$1$. 
\end{lemma}

\begin{remark} \label{remark about pure inner forms}
	The authors of~\cite{KMSW} work with a refinement of the notion of inner form. Recall that isomorphism classes of inner forms of~$G$ are in bijection with~$H^1(\Gamma_F, G^{\text{ad}})$. In addition to this, they introduce the notion of pure inner form, a triple consisting of~$G$, the map~$\xi$, and a cocycle~$z \in Z^1(\Gamma_F, G)$ compatible with the inner twist~$\xi$. The map sending a pure inner form to~$z$ induces a bijection between isomorphism classes of pure inner forms and~$H^1(\Gamma_F, G)$. Inner forms of $U(N)$ which can be realized as pure inner forms are those coming from a Hermitian space, i.e. precisely the groups we work with. We will point out dependency on~$z$ whenever it appears: in the normalization of transfer factors, and the pairings associated to local Arthur packets. Due to our rather rudimentary use of the stable trace formula, the choice of pure inner form does not affect our results.
\end{remark}

\subsubsection{$L$-groups} \label{section on L-groups of unitary groups} Throughout, we will work with the Weil group version of the $L$-group, primarily because it is well-suited to our description of local parameters. In terms of the actual definition of the $L$-group, this choice is purely cosmetic as the Galois actions involved will always factor through a quotient of order at most 2.

For $G/F$ with $F$ either local or global, fix a root datum.  The $L$-group of $G$ is a semidirect product \[ ^LG =  \hat{G} \rtimes W_F. \] Here the group~$\hat{G}$ is the complex dual group of~$G$, and the action of~$W_F$ on~$\hat{G}$ is induced by the Galois action on the root datum of~$G$. As a consequence, if~$G$ is split then~$^LG = \hat{G} \times W_F$, and in particular,~$^LGL_N(F) = GL_N(\BC) \times W_F.$ If~$G'/F$ is an inner form of $G$ then by definition $G'(\bar{F}) \simeq G(\bar{F})$ and the corresponding Galois actions differ by an inner automorphism. These induce isomorphisms of root data and Galois actions, and ${^L}G \simeq {^L}G'$.

For $F$ global, we will sometimes abuse notation and denote by~$^LG_v$ the $L$-group of the base change of~$G$ to a completion $F_v$. In this situation, the embedding~$W_{F_v} \to W_F$ induces a map~$^LG_v \to {^L}G$ which restricts to the identity on $\hat{G}$. 

The $L$-group of~$U(N)$ (and of any inner form) is defined as \[ ^LU(N) = GL_N(\BC) \rtimes W_F \] where $W_F$ acts through the order two quotient~$\Gamma_{E/F}$. The non-trivial element~$\sigma \in \Gamma_{E/F}$ acts by~$ \sigma(g) = \Phi_N^{-1}g^{-t}\Phi_N$ of~$GL_N$, where~$\Phi_N$ is as in~\eqref{Matrix Phi_N}.

\subsubsection{Morphisms of $L$-groups} \label{section on morphisms of L-groups}

A morphism of~$L$-groups, or $L$-morphism, is a continuous morphism \[ \eta: {^L}H \to {^L}G \] which commutes with the projections onto~$W_F$. In practice, all morphisms of $L$-groups considered here will be admissible, i.e. induced by an algebraic map $\hat{H} \to \hat{G}$ and such that the image of elements of $W_F$ are semisimple. 

Denote the Weil restriction~$\Res^E_FGL_N$ by~$G(N)$.  In particular, $G(N)(F) = GL_N(E)$ and $\GN(E) \simeq  GL_N(E) \times GL_N(E)$. As such, the connected component $\widehat{G(N)}$ of~$^L \GN$ is the product of two copies of $GL_N(\BC)$, and $W_F$ acts through~$\Gamma_{E/F}$ via the automorphism that interchanges the two factors. Many objects associated to a unitary group~$U(N)$ depend on a choice of embedding of~$L$-groups from~${^L}U(N)$ to~${^L}\GN$.  

To define the~$L$-embedding~$\emb: {^L}U(N) \to {^L}\GN$, recall the characters~$\chi_\kappa$ from~\ref{section on notation}. If $F$ is global, we will use these characters, and if~$F = F_v$ is local, we will momentarily also denote by~$\chi_{\kappa}$ the restriction of~$\chi_\kappa$ to~$E_v^\times$. Let $I_N$ be the identity matrix. For each~$\kappa \in \{\pm1\}$ we define~$\eta_{\kappa}$ as \begin{align} \begin{split} \label{the choice of bas change morphism}
	\emb(g \rtimes 1) &= (g, {^t}g^{-1}) \rtimes 1, \quad g \in \hat{G} \\
	\emb(I_N \rtimes \sigma) &= (\chi_\kappa(\sigma)I_N, \chi_{\kappa}^{-1}(\sigma)I_N) \rtimes \sigma, \quad \sigma \in W_E \\ 
	\emb(I_N \rtimes w_c) &= (\kappa \Phi_N, \Phi_N^{-1}) \rtimes w_c.
	\end{split}
\end{align} 
We consider a second class of $L$-embeddings $\embend: {^L}(U(N_1) \times... \times U(N_r)) \to {^L}U(N)$,  for ~$\sum N_i = N$ into $^{L}U(N)$. Put $\kappa_i = (-1)^{N-N_i}$ for each $i$, and let $\underline{\kappa} = (\kappa_1,...,\kappa_r)$. Given $\underline{\chi}$ with signature $\underline{\kappa}$, and for a choice of $w_c$ as above, the embedding $\embend$ is defined as 
\begin{align}\begin{split}
\label{embedding of endoscopic groups}
	\embend(g_1,...,g_r \rtimes 1) &= \text{diag}(g_1,...,g_r)\rtimes 1, \quad g_i \in GL_{N_i}(\BC)\\
	\embend(I_{N_1},...,I_{N_r} \rtimes \sigma) &= \text{diag}(\chi_{\kappa_1}(\sigma)I_{N_1},...,\chi_{\kappa_r}(\sigma)I_{N_r}) \rtimes \sigma, \quad \sigma \in W_E \\ 
	\embend(I_{N_1},...,I_{N_r} \rtimes w_c) &= \text{diag}(\kappa_1\Phi_{N_1},...,\kappa_r\Phi_{N_r}) \cdot \Phi_N^{-1} \rtimes w_c.
	\end{split} 
\end{align}

Note that the composite embedding $\emb \circ \embend$ gives an embedding \begin{equation}\label{eqn embedding into GLN} \bcend: 
{^L}(U(N_1) \times ... \times U(N_r)) \to {^L}\GN
\end{equation} with signature $\kappa \cdot \underline{\kappa} = (\kappa \kappa_1, .... ,\kappa\kappa_r)$. 

\begin{remark}
	The need to consider several embeddings depending on $\underline{\kappa}$ stems from the possibility that parameters for the pair $(U(N), \eta_{+})$ may factor through different embeddings of the products of groups $U(N_i)$ associated to different signs. 
\end{remark}

\subsection{Parameters}
We introduce the discrete automorphic spectrum of  a unitary group $G$, and the local and global parameters which will classify the (constituents of) automorphic representations, following \cite{A}, \cite{KMSW}, and \cite{Mok}. 
\subsubsection{Automorphic Representations} \label{subsection on automorphic representations}
Let $G/F$ be a reductive group. Fix a closed subgroup $\fX \subset Z_G(\BAF)$ and a maximal compact subgroup of $K$ of~$G(\BAF)$, which in turn determines maximal compact subgroups $K_v$ of~$G_v = G(F_v)$ for any $v$. We consider the right-regular representation of~$G(\BAF)$ on\[\Aut(G(F) \dom G(\BAF), \omega), \] the discrete part of the space of square-integrable functions which transform by $\omega$ under the action of $\fX$. We will omit the $\omega$ when we allow for any central character. In our initial cases of interest, $G/F$ will be an anisotropic inner form of~$U_{E/F}(N)$, the group~$\fX$ will be the full center, the central character~$\omega$ will be trivial, and the entire automorphic spectrum will be discrete. However for induction purposes we will consider arbitrary central character data~$(\fX, \omega)$ and allow for~$L^2(G(F)\dom G(\BAF), \omega)$ to have a continuous part.  The discrete spectrum decomposes as \[ \Aut(G(F) \dom G(\BAF)) = \bigoplus m(\pi) \pi \] where $m(\pi)$ denotes the multiplicity of $\pi$, and the irreducible constituents are automorphic representations. Each $\pi$ is a restricted tensor product~$\pi = \ten_v'\pi_v$ with each~$\pi_v$ an irreducible
admissible unitary representation
of~$G_v$. All but finitely many $\pi_v$ are spherical with respect to $K_v$.  The representation $\pi_v$ is said to be tempered if its $K_v$-finite matrix coefficients belong to~$L^{2+\epsilon}(G_v)$ for all $\epsilon>0$. 

After fixing a maximal compact subgroup $K_\infty$ of $G_\infty$, we replace $\pi_\infty$ by the dense subspace of $K_\infty$-finite smooth vectors, which we view as an admissible $(\fg_\infty,K_\infty)$-module. This is no loss of information since unitary admissible representations are determined by their underlying $(\fg_\infty,K_\infty)$-modules \cite[9.2]{Kn01}.

\subsubsection{Local Langlands Parameters} \label{subsection on local Langlands}
 Let $F$ be a local field with Weil group~$W_F$. The Langlands group~$L_F$ of~$F$ is defined as  \[L_F := \begin{cases}
W_F & F \text{ is archimedean} \\ 
W_F \times SU(2,\BR)& F \text{ is non-archimedean.}
\end{cases}\] 

A local Langlands parameter  for the reductive group $G/F$ is a continuous  homomorphism $ \phi: L_F \to {^LG}$ satisfying certain conditions (see \cite{Bo79} for a discussion):

\begin{itemize}
	\item[(i)] The map~$\phi$ commutes with the  projections~$L_F \to W_F$ and~${^L}G \to W_F$. 
	\item[(ii)] In the non-archimedean case, the restriction $\phi\mid_{SU(2,\BC)}$ is algebraic. 
	\item[(iii)] The image of~$W_F$ under~$\phi$ consists of semisimple elements of~${^L}G$. 
	\item[(iv)] If the image of $\phi$ in $\hat{G}$ factors through a parabolic subgroup of $\hat{G}$, then this parabolic subgroup must be the dual $\hat{P}$ of a parabolic subgroup $P$ of $G$. 
\end{itemize}

Continuous maps that satisfy condition (i) are known as~$L$-homomorphisms. If they additionally satisfy (ii)-(iii) they are called \emph{admissible.} If they satisfy (iv), they are called \emph{relevant}, or \emph{G-relevant}. Finally, we say that~$\phi$ is \emph{bounded} if~$W_F$ has bounded image in~$\hat{G}$. We will denote the collection of $\hat{G}$-conjugacy classes of Langlands parameters for~$G$ by~$\Phi(G)$.

\subsubsection{Local Arthur Parameters}\label{How local parameters of unitary groups are the image of base change. } In order to describe the non-tempered spectrum of~$G$, we consider enhancements of Langlands parameters known as Arthur parameters. These are admissible~$L$-homomorphisms \[ \psi: L_F  \times SL_2(\BC) \to {^LG}\] such that~$\psi \mid_{L_F}$ is bounded. We denote the set of~$\hat{G}$-conjugacy classes of Arthur parameters by~$\Psi(G)$. We refer to the~$SL_2(\BC)$ factor in the above product as the ``Arthur~$SL_2$", and say that~$\psi$ is bounded if it restricts trivially to the Arthur~$SL_2$.  

Each Arthur parameter~$\psi$ determines a Langlands parameter~$\phi_{\psi}$ as follows. Recall (eg. \cite{Ta79}) that the Weil group~$W_F$ is naturally equipped with a norm homomorphism~$\mid \cdot \mid$ to~$\BC^\times$. Then $\phi_{\psi}$ is defined as the composition\[ \phi_{\psi} : W_F \to {^L}G, \quad \phi_{\psi}(\sigma) = \psi \left( \sigma, \tbt{|\sigma|^{1/2}}{0}{0}{|\sigma|^{-1/2}}\right).  \]

We now give a more detailed description of local Arthur parameters in the case where $G = U(N)$, following Section 2.2 of Mok \cite{Mok}.  Specifically, we use the map $\eta_{\kappa}$ introduced in \ref{section on morphisms of L-groups} to realize $\Psi(U(N))$ as a set of $N$-dimensional representations satisdying an appropriate self-duality condition.

We first describe a natural bijection between $\Psi(G(N))$, and $\Psi(GL_N(E))$. To produce an element of~$\Psi(\GN)$, one starts with $\psi \in \Psi(GL_N(E))$, i.e. an admissible~$N$-dimensional representation of~$L_E \times SL_2(\BC)$,  and promotes it to a $L$-morphism~$\psi':L_F \times SL_2(\BC) \to {^L}\GN$ by choosing $w_c \in W_F \setminus W_E$ and defining \begin{align*} \psi'(\sigma,g) &= (\psi(\sigma,g), \psi^c(\sigma ,g)) \rtimes \sigma, \quad (\sigma, g) \in L_E \times SL_2(\BC) \\ 
	\psi'(w_c) &= (\psi(w_c^2), I_N) \rtimes w_c,
 \end{align*} where $\psi^c(\sigma, g) = \psi(w_c^{-1}\sigma w_c,g)$. The resulting bijection $\Psi(G(N)) \simeq \Psi(GL_N(E))$ is independent of the choice of $w_c$.  Moreover, if~$\psi^c \simeq \psi^\vee$ where~$\psi^\vee$ is the contragredient of~$\psi$, then~$\psi$ is called \emph{conjugate self-dual}. More precisely, the parameter $\psi$ is conjugate self-dual of parity~$\pm 1$, depending on the parity of the resulting bilinear form.

The map~$\emb$ introduced in \eqref{the choice of bas change morphism} then induces a mapping \begin{align}\begin{split}
\label{map giving unitary parameters}
{\emb}_*: \Psi(U(N)) &\to \Psi(\GN) \simeq \Psi(GL_N(E))
\end{split} 
\end{align} which is shown by Mok, following work of Gan-Gross-Prasad~\cite{GGP12}, to be an injection whose image consists precisely of the subset of $\Psi(GL_N(E))$ of conjugate self-dual representations of parity~$(-1)^{N+1}\kappa$. 

\subsubsection{Global Arthur Parameters} \label{SectionOnGlobalArthurParameter} In lieu of global parameters, Arthur~\cite[\S 1.4]{A} introduces formal objects realized by combining cuspidal automorphic representations of~$GL_{N}$ and representations of the Arthur~$SL_2$. Echoing the local discussion, global Arthur parameters are first defined in terms of~$\GN$, and Arthur parameters for~$U(N)$ are the ones factoring through a fixed embedding of~$L$-groups.  

A global Arthur parameter for~$GL_N$ is an unordered sum \[ \psi^N = \boxplus_i \psi^{N_i}_i, \quad \psi^{N_i}_i = \mu_i \boxtimes \nu(m_i). \] Here $\mu_i$ is a cuspidal automorphic representation of $GL_{n_i}(\BA_E)$ and $\nu(m_i)$ is the irreducible $m_i$-dimensional representation of~$SL_2(\BC)$, with~$m_in_i = N_i$ and~$\sum_i N_i  = N$. Departing from our references, we immediately restrict our attention to the set of Arthur parameters such that the~$\psi_i^{N_i}$ are pairwise distinct: we denote this set $\Psi(N)$ instead of $\Psi_{\text{ell}}(N)$. 
The collection $\Psi(N)$ contains a distinguished subset $\Psi_{\text{sim}}(N)$ of simple parameters with a unique summand~$\psi^{N}$. Following the theorem of M\oe glin--Waldspurger \cite{MW89}, this subset $\Psi_{\text{sim}}(N)$ parameterizes the discrete spectrum of $GL_N$.

We now give the construction of global Arthur parameters for a quasi-split unitary group~$G = U(N)$, following Section 1.3.4 of~\cite{KMSW}.  We start by restricting our attention to the set~$\tilde{\Psi}(N) \subset \Psi(N)$ of parameters for which each of the~$\mu_i$ is \emph{conjugate self-dual}, i.e.  satisfies~$\mu_i = \bar{\mu_i}^\vee$ where~$\bar{\mu} = \mu \circ \sigma$ and $\sigma \in \GalEF$ is nontrivial.

To record the parameter in relation to the embedding $\emb$, we introduce the group $\mL_\psi$. If $\psi^N$ decomposes as a sum of $\mu_i \boxtimes \nu(m_i)$, we associate to each index a pair $(U_{E/F}(n_i), {\emb}_i)$  as in \ref{section on morphisms of L-groups}. Here the choice of sign~$\kappa_i$ is determined by~$\mu_i$. Then $\mL_\psi$ is the fiber product $\mL_\psi =  \prod_i  ({^L}U_{E/F}(n_i) \to W_{F})$. There is a natural map~$\tilde{\psi}^N: \mL_\psi \times SL_2(\BC) \to {^L}\GN$ given by the direct sum \[ \tilde{\psi}^N = \oplus ({\emb}_i \ten \nu(m_i)). \]

A global Arthur parameter for~$(U_{E/F}(N), \emb)$ is then defined as a pair~$\psi = (\psi^N, \tilde{\psi})$ where~$\psi^N \in \tilde{\Psi}(N)$, and \[ \tilde{\psi}: \mL_\psi \times SL_2(\BC) \to {^L}U_{E/F}(N) \] is an~$L$-homomorphism such that~$\emb \circ \tilde{\psi} = \tilde{\psi}^N$. It is useful to remember that~$\psi^N$ encodes the arithmetic information of the automorphic representations of~$GL_{n_i}$, and that~$\tilde{\psi}$ is an actual homomorphism. As such, we can (and will) discuss the centralizer of the image of~$\tilde{\psi}$. Two Arthur parameters are equivalent if the $\tilde{\psi}$ are~$\hat{U}(N)$-conjugate, and we denote the set of equivalence classes of $\psi$ as above by~$\Psi(U(N), \emb)$. Note that we have again broken off from our references in the choice of notation: our set~$\Psi(U(N), \emb)$ is the one that the authors of~\cite{KMSW} denote $\Psi_2(U_{E/F}(N), \emb)$. Finally, note that the map $\eta_{\kappa,*}$ sending~$\psi$ to  $\psi^N$ is an injection: this allows us to view~$\Psi(U(N), \emb)$ as a subset of~$\Psi(N)$. If $(G, \embend)$ is a product as in~\eqref{eqn embedding into GLN}, we can similarly define~$\Psi(G, \embend)$. Via the block-diagonal embedding~$\prod_i GL_{N_i} \hookrightarrow GL_N$, we can identify $\Psi(G, \embend) \simeq \prod_i \Psi(U(N_i), \eta_{\kappa_i})$. 

\begin{remark}\label{our restriction on the parameter}
	We have made two constraints on the set of parameters under consideration here which bear highlighting. We require: \begin{itemize}
		\item[(i)] that the irreducible summands $\psi_i$ be pairwise distinct. In Mok's description of the parameters in \cite[\S 2.4]{Mok} this amounts to requiring that all the $l_i = 1$. 
		\item[(ii)] that each irreducible summand be conjugate self-dual. This is stricter than requiring~$\psi$ to be conjugate self-dual since we could have had~$\mu_i^\vee \simeq \mu_j$. 
	\end{itemize}
	Parameters satisfying these conditions are called \emph{elliptic}. These restrictions will give us control on the group $\mS_{\psi}$ to be introduced below, whose characters determine which products of local representations occur in the discrete spectrum. It is also the case that only the parameters in the set which we denote by~$\Psi(U(N), \emb)$ correspond to packets whose members actually appear in the decomposition of~$\Aut$, although this fact is far from obvious and is one of the main theorems in \cite{Mok} and \cite{KMSW}. Following this result, global elliptic parameters are also called \emph{square-integrable.} 
\end{remark}
\subsubsection{Localization}
We now describe how a global Arthur parameter~$\psi \in \Psi(U(N), \emb)$ gives rise to local Arthur parameters~$\psi_v$  at each place~$v$. Each cuspidal representation~$\mu$ of~$GL_N$ factors as a restricted tensor product~$\mu = \ten' \mu_v$ over places~$v$ of~$F$. The~$\mu_v$ are admissible representations of~$GL_N(F_v)$. The local Langlands correspondence  for~$GL_N$ ~\cite{HT01,He00,Sc13} associates to $\mu_v$ a parameter $\phi_{\mu_v} \in \Phi(GL_N)$. Following \cite{A}, we define the localization of~$\psi$ at~$v$ as the direct sum \[\psi_v = \oplus_i \psi_{v,i} , \quad \psi_{v,i} = \phi_{\mu_{v,i}} \ten \nu(m_i).\] 
These localizations a priori only belong to $\Psi(\GN)$. The fact that they are in the image of the map \eqref{map giving unitary parameters} is one of the central theorems of \cite{Mok}. 

\subsubsection{Parameters of Inner Forms}
Let $(G, \xi)$ be an inner form of $G^* = U(N)$. A local Arthur parameter for~$G$ is simply a~$G$-relevant parameter for $U(N)$, see Section~\ref{subsection on local Langlands}. Globally, a parameter $\psi \in \Psi(G^*, \emb)$ is $G$-relevant if it is so everywhere locally \cite[\S 1.3.7]{KMSW}. We denote by $\PsiG$ the collection of $G$-relevant parameters in $\Psi(G^*, \emb)$. In summary, we have the following chain of inclusions: \[ \PsiG \subset \Psiqs \subset \tilde{\Psi}(N) \subset \Psi(N),\] where the parameters in $\tilde{\Psi}(N)$ are conjugate self-dual, those in $\Psiqs$ factor through the embedding $\emb$, and those in $\PsiG$ are additionally $G$-relevant.

 \subsubsection{Parameters and Conjugacy Classes} \label{parameters determine conjugacy classes}We attach families of conjugacy classes to objects introduced above, following \cite[\S 1.3]{A}. For~$F$ global,~$G$ reductive, and any finite set~$S$ of places of~$F$ containing the archimedean ones, let~$\mC^S(G)$ denote the set of collections~$c = \{c_v\}_{v \notin S}$, where each~$c_v$ is a semisimple conjugacy class in~$\hat{G}$. For two sets~$S$ and~$S'$, let~$c \sim c'$ if~$c_v = c'_v$ for almost all~$v$. Denote the set of such equivalence classes by~$\mC(G)$. As we did for parameters, let~$\mC(N):=\mC(GL_N)$. We associate elements of~$\mC(G)$ to automorphic representations~$\pi$ of~$G$. Factoring~$\pi = \ten'_v \pi_v$, let~$c(\pi) = \{c(\pi_v)\} \in \mC(G)$ be the Satake parameters of all the unramified~$\pi_v$. Note also that an $L$-embedding $\eta: {^L}G \to {^L}G(N)$ such as those introduced in~\ref{section on morphisms of L-groups} gives rise to a map $ \eta_*: \mC(G) \to \mC(N)$. 

When~$G = GL_N$ one can associate an element of~$\mC(N)$ to each~$\psi \in \Psi(N)$. Starting with simple parameters~$\psi \in \Psi_{\text{sim}}(N)$, use the recipe for the representation~$\pi_{\psi}$ prescribed by Moeglin-Waldspurger's theorem~\cite{MW89} and let~$c(\psi) := c(\pi_{\psi})$. If~$\psi$ is not simple, apply the process to its simple constituents and associate to~$\psi$ the conjugacy class coming from the diagonally embedded product of the~$GL_{N_i}$ inside of~$GL_N$. This produces a mapping \[\Psi(N) \to \mC(N), \quad \psi \mapsto c(\psi)\] which is injective, following  Jacquet-Shalika \cite{JS81}. Denote its image by $\mC_{aut}(N)$. 

\subsubsection{Stabilizers and Quotients}\label{section where we compute the stabilizers}
For $\psi$ either local or global, we have \[  S_\psi := \mathrm {Cent} (\mathrm{Im}(\psi), \hat G), \quad 
\bar{S}_\psi := S_\psi /Z(\hat G)^{W_F},  \quad 
\mS_\psi:= \pi_0(\bar{S}_\psi).
\]  As mentioned previously, when~$\psi$ is global then~$\text{Im}(\psi)$ really means~$\text{Im} (\tilde{\psi})$.  Localization of parameters~$\psi \mapsto \psi_v$ induces a mapping~$\mS_{\psi} \to \mS_{\psi_v}$.  When~$G$ is unitary, the groups~$\mS_{\psi}$ can be readily computed, as the four authors do in \cite[p.63]{KMSW}. In particular, for $F$ global and~$\psi \in \Psiqs$ decomposing as~$\psi = \boxplus_{i=1}^r \psi_i$, we have \begin{equation} \label{size of S_psi} \mS_{\psi} = (\BZ/2\BZ)^{r-1}.  \end{equation}
The reader who looks at the computations in~\cite{KMSW} will notice that this is the point where we use the assumptions from Remark~\ref{our restriction on the parameter}. Finally, we introduce the element  
\begin{equation} \label{spsi} s_\psi := \psi\left( 1, \tbt{-1}{0}{0}{-1} \right) \in S_\psi. \end{equation}  We will sometimes conflate $s_\psi$ and its image in the quotient $\mS_{\psi}$. 
\begin{remark}\label{we will not use S-becare}
	The authors of  \cite{KMSW} work with the centralizer quotient~$S^\natural_\psi$, which agrees with~$\mS_{\psi}$ for~$G$ local and unitary. If the local group~$G_v$ is isomorphic to~$GL_{N,v}$ (the only possibility for us at split places, since our unitary groups arise from Hermitian forms) then~$S^\natural_\psi \simeq \BC^\times$. However, if~$G_v = GL_{N,v}$, then only the trivial character of~$S^\natural_\psi$ arises in the character identities, as will be discussed in Section~\ref{Section on Local packets for GLN}. Thus there is no loss in working instead with the group~$\mS_{\psi} = \{1\}$. In the global situation, the characters of~$S^\natural_\psi$ that arise all factor through~$\mS_{\psi}$~\cite[p. 89]{KMSW}. Note that we follow Arthur's convention and use the notation~$\mS_{\psi}$ instead of~$\bar{\mS}_\psi$ as in  \cite{KMSW}. 
\end{remark}

\subsubsection{Epsilon Factors} \label{sectionon epsilon factors} The last invariant attached to a global parameter~$\psi$ is the character~$\epsilon_\psi$ of~$\mS_{\psi}$, defined by Arthur in~\cite[\S 1.5]{A}.
The definition involves the symplectic root number~$\epsilon(1/2, \mu_\alpha)$ of an automorphic $L$-function~$L(s,\mu_\alpha)$ for a product of general linear groups, obtained by composing $\psi$ with the adjoint representation.  As such~$\epsilon_{\psi}$ encodes arithmetic data in the decomposition of~$\Aut(G(F) \dom G(\BAF))$. Note that~$\epsilon_\psi$ only depends on~$\psi$ and in particular is independent of the inner form of~$G^*$ under consideration, as discussed in~\cite[p.89]{KMSW}.
\subsection{Endoscopic Data} \label{Section on Endoscopic data.}
An endoscopic datum for~$G/F$ is a triple~$(\xi, H, s)$ where

\begin{itemize}
	\item[-] $s$ is a semisimple element of $\hat{G}$,
	\item[-]$H/F$ is a connected, quasisplit group,
	\item[-]$\xi : {^L}H \to {^L}G$ is an $L$-embedding. 
\end{itemize}  

The triple must satisfy certain conditions, see \cite[\S 1.1.1]{KMSW}, including that~$\xi(\hat{H})$ is the connected component of the centralizer of~$s$ in~$\hat{G}$. We will work only with \emph{elliptic} endoscopic data, characterized by the requirement that $\xi(Z(\hat{H})^{W_F})^0 \subset (Z(\hat{G}))^{W_F}$. As such, we denote the set of conjugacy classes of elliptic endoscopic data for $G$ by $\mE(G)$, dropping the ``ell" subscript appearing in our references. An endoscopic datum of $G$ for which ${^L}H \not\simeq  {^L}G$ will be called \emph{proper}.  We will frequently abuse notation and refer to $H$ as a stand-in for the full datum, and denote the other two elements of the triple by $\xi_H$ and $s_H$. Lastly, we will also use the formalism of endoscopic data for our unitary groups and denote by~$\tilde{\mE}(N)$ the set of pairs consisting of a product~$G$ of quasisplit unitary groups together with the~$L$-embedding~$ \xi = \eta_{ \kappa, \underline{\kappa}}$ from Section~\ref{section on morphisms of L-groups}, and by~$\tilde{\mE}_{\text{sim}}(N)$ the subset of~$\tilde{\mE}(N)$ for which~$G = U(N)$. 

For any inner form~$\U$ of~$U_{E/F}(N)$,  the set~$\mE(G)$ consists of pairs\[(H, \xi) = (U(N_1) \times U(N_2), \embend), \quad N_1,N_2 \geq 0, \; N_1+N_2 = N,\] where $ \embend$ was defined in \eqref{embedding of endoscopic groups}. The signature $\underline{\kappa} = ((-1)^{N-N_1}, (-1)^{N-N_2})$ depends on the respective ranks of the groups. The equivalence class of endoscopic data is then uniquely determined by $N_1$,  see \cite[\S 2.4]{Mok}. 

\subsubsection{Endoscopic Data and Parameters} \label{section on endoscopic data and parameters} 
Let $F$ be global and $G/F$be unitary and $\psi = (\psi^N, \tilde{\psi}) \in \Psiqs$ be an Arthur parameter. Let~$(H,\xi_H,s_H) \in \mE(G)$, and let~$\psi^H = (\psi^{N,H}, \tilde{\psi}^H) \in \Psi(H,\eta_{\kappa}\circ \xi_H)$ be an Arthur parameter for $H$ satisfying~$\psi^N = \psi^{N,H}$ and $\tilde{\psi} = \xi_H \circ \tilde{\psi}^H$. In this situation, we will abuse notation and write that $\psi = \xi_H\circ \psi^H$.  Since $s_H$ commutes with~$H$, it also commutes with the image of~$\tilde \psi$.   We thus get a mapping
\begin{equation}  \label{map sending H-parameter to G-parameter}  (H, \psi^H) \mapsto (\xi_H\circ\psi^H, s_H) \end{equation} from the set of pairs~$(H, \psi^H)$ onto the set of pairs consisting of a parameter~$\psi$ for~$G$ and an element $s$ of the centralizer~$S_\psi$. The importance of the quotient~$\mS_{\psi}$ comes from the fact that for each~$\psi$, the map~\eqref{map sending H-parameter to G-parameter} descends to a bijection between~$\mS_{\psi}$ and the set of endoscopic data such that~$\psi$ factors through~$\xi_H$. We state this result below, under simplifying assumptions:~$G$ global unitary and~$\psi$ square-integrable.  

\begin{lemma} \label{lemma on endoscopic bijection} Let~$F$ be global and~$G^* = U_{E/F}(N)$. Let~$\psi=(\tilde \psi,\psi^N) \in \Psiqs$. The map~\eqref{map sending H-parameter to G-parameter} induces a bijection \[(H, \psi^H) \leftrightarrow (\psi, s)\] where the left-hand side runs over pairs where~$H$ stands in for an endoscopic datum~$(H, \xi,s)$ and $\psi^H = (\tilde \psi^H, \psi^{N,H}) \in \Psi(H,\eta_{\kappa}\circ \xi_H)$ with~$\psi^N = \psi^{N,H}$ and~$\tilde{\psi} = \xi \circ \tilde\psi^H$, and the right-hand side runs over elements of~$\mS_{\psi}$. 
\end{lemma}
\begin{proof}
	The proof occupies Section 1.4 of \cite{KMSW}, and the above statement is a reformulation of Lemma 1.4.3 therein. The square-integrability assumption on~$\psi$ implies that $S_\psi$ and a fortiori $\bar{S}_\psi$ are finite. Thus  $\bar{S}_\psi = \mS_{\psi}$ and we use the latter. 
\end{proof}

\subsection{Packets} \label{section on packets} Here, we introduce~$A$-packets of representations associated to Arthur parameters, and the character identities relating their traces to those of corresponding representations for endoscopic groups. 
\subsubsection{Local Arthur Packets}\label{local arthur packets - definition} Let $(G, \xi)$ be a unitary group over a local field. The main local results of Mok \cite[Theorem 2.5.1]{Mok} and Kaletha-Minguez-Shin-White \cite[Theorem 1.6.1]{KMSW} associate to each Arthur parameter $\psi \in \PsiG$ a finite set $\Pi_\psi$ of irreducible unitary representations of~$G(F)$ called a local Arthur packet. This packet~$\Pi_{\psi}$ is empty if $\psi$ is not relevant, and contains only tempered representations when~$\psi$ is bounded. Each nonempty~$\Pi_{\psi}$ is equipped with a pairing \begin{equation} \label{local pairing} \pa{\;}{\;}: \Spsi \times \Pi_\psi \to \{\pm 1\}. \end{equation}In this way, every~$\pi \in \Pi_\psi$ gives rise to a character of~$\Spsi$. Unramified representations correspond to the trivial character.  The pairing depends on the triple $(G, \xi, z)$ realizing $G$ as a pure inner twist, as discussed in Remark \ref{remark about pure inner forms}.

For~$F$ archimedean, all~$\pi \in \Pi_{\psi}$ have the same infinitesimal character. We recall how to compute it from ~$\phi_{\psi}$ following~\cite{NP20Cohomo}. The group~$W_\BR$ is an extension of~$\BC^\times$ by the group~$\langle \sigma \rangle$ of order~$2$. For each~$\psi$, there is a torus~$\hat{T} \in \hat{G}$ such that \[ \phi_{\psi} \mid_{\BC^\times}(z) = z^{\mu}\bar{z}^{\nu}, \quad \mu, \nu \in X_*(\hat{T}). \] The infinitesimal character of the representations~$\pi \in \Pi_{\psi}$ is then identified with~$\mu \in X_*(\hat{T}) \simeq X^*(T)$ via the Harish-Chandra isomorphism.

\begin{lemma} \label{lemmainfcharremainsregular}
Let ~$\psi \in \Psi(G)$ be an archimedean Arthur parameter with regular infinitesimal character, and let~$H \in \mE(G)$ be such that~$\psi = \embend \circ \psi^H$ for~$\psi^H \in \Psi(H)$. Then the infinitesimal character of~$\psi^H$ is also regular.
\end{lemma}
\begin{proof}
By assumption,~$\psi((z,1),I) = z^\mu\bar{z}^\nu$ and the weights appearing in~$\mu$ are distinct. The parameter~$\phi$ factors through~$\embend: {^L}H \to {^L}G.$ Referring to~\eqref{the choice of bas change morphism}, the restriction of $\embend$ to $\BC^\times \subset W_\BR \subset {^L}H$ is trivial, since it factors through~$\chi_{\kappa}$, which takes values in~$\pm 1$. Thus the weights of the $z$-part of~$\phi^H\mid_{\BC^\times}$ are also distinct, and the infinitesimal character of the corresponding packet is regular. 
\end{proof}
For any local~$F$, we record a result initially proved by Mok about the central character of the representations in the packet~$\Pi_{\psi}$ for the quasisplit group~$G^*$.  
\begin{proposition}[Proposition 1.5.2, 2.\cite{KMSW}] \label{proposition on central character} The Langlands parameter of the central character~$\omega_\pi: Z(G^*)(F) \to \BC^\times$  of any~$\pi \in \Pi_{\psi}$ is given by the composition \[ L_F \xrightarrow[]{\phi_{\psi}}{^L}G^* \xrightarrow[]{(\det \rtimes \id) \circ \emb} \BC^\times \rtimes W_F. \]
\end{proposition}

\subsubsection{Global Arthur Packets}
 Let~$\psi \in \PsiG$ be global with localizations~$\psi_v$.  The global Arthur packet~$\Pi_{\psi}$ is then defined as \[  \Pi_{\psi} = \left\{ \pi = \ten_v \pi_v  \mid \pi_v \in \Pi_{\psi_v}, \; \ip{\cdot}{\pi_v}_{\psi_v} = 1 \text{ for almost all $v$} \right\}. \] It is equipped with a pairing \begin{equation}\label{global pairing} \ip{\;}{\;}_\psi:  \mS_{\psi} \times \Pi_{\psi} \to \{\pm 1\}, \quad \ip{\;}{\pi}_\psi = \prod_v \ip{\;}{\pi_v}_{\psi_v} \end{equation} determined by the maps $\mS_{\psi} \to \mS_{\psi_v}$ induced by localization.  We note once again that this pairing depends on the full inner twist $(G, \xi)$. However, the local dependence on the pure inner twist, i.e. the dependency on the cocycle $z$ appearing in the local definition of the pairing, cancels out globally. This is detailed in \cite[\S 1.7]{KMSW}.

\subsubsection{Test Functions} \label{SectionOnTestFunctions}
Continuing with~$F$ global,  we fix a maximal compact subgroup~$K$ of~$G(\BAF)$. The group~$K$ determines a maximal compact subgroup~$K_v \subset G_v$ at each place~$v$: we choose~$K$ so that $K_v$ is hyperspecial at all the unramified~$v$. We also fix for each~$v$ a Haar measure~$\mu_v$ on~$G_v$ satisfying~$\mu_v(K_v) = 1$, and a corresponding measure $\mu = \prod_v \mu_v$ on~$G(\BAF)$.  WeThe local Hecke algebra~$\mathscr{H}(G_v)$ consists of smooth, compactly supported, left and right~$K_v$-finite functions on~$G_v$. We will call its elements local test functions. The global Hecke algebra is the restricted product~$\mathscr{H}(G) = \ten_v' \mathscr{H}(G_v)$: it consists of smooth, compactly supported,~$K$-finite functions.  Each such test function is a finite sum of factorizable test functions of the form~$f = \prod_v f_v$, where each~$f_v \in \mathscr{H}(G_v)$ and all but finitely many~$f_v$ are the characteristic function of~$K_v$.  

For $\pi_v$ a smooth, admissible representation of $G_v$, each $f_v \in \mH(G_v)$ gives rise to an operator $\pi_v(f_v)$ on th underlying vector space of $\pi_v$, defined as follows: \[  \pi_v(f_v)(x) =   \int_{G_v} f_v(g) \pi_v(g)(x) d\mu_v. \] This operator is of trace class, and we denote its trace by~$\tr \pi_v(f_v)$. Likewise globally, the algebra~$\mathscr{H}(G)$ acts on~$\Aut(G(F) \dom G(\BAF))$ and on its irreducible constituents~$\pi$. We denote the trace of convolution by~$f$ by~$\tr R(f)$ (when considering the right-regular representation on~$\Aut(G(F), G(\BAF))$) or by~$\tr \pi(f)$ (when~$f$ acts on~$\pi$ irreducible). 

\subsubsection{Stable Distributions and Transfer}
We introduce stable distributions on the local and global Hecke algebras, following Sections 3.1 and 4.2 of~\cite{Mok} respectively. Let~$\gamma \in G(F_v)$ and  let~$G(F_v)_{\gamma}$ be its centralizer. For~$f \in \mathscr{H}(G(F_v))$, let~$ f_{G(F_v)}(\gamma) := \int_{G(F_v)/G(F_v)_{\gamma}} f(g\gamma g^{-1})d\mu_v $ be the \emph{orbital integral} associated to~$\gamma$ and~$f$. It only depends on the $G(F_v)$ conjugacy class of $\gamma$. 

We now introduce transfer, which makes use of  \emph{stable conjugacy classes}: the union of the finitely many conjugacy classes of~$G(F_v)$ that are~$G(\bar{F_v})$-conjugate. Let~$G(F_v)$ first be a quasisplit unitary group. Each stable conjugacy class~$\delta$ gives rise to a linear functional \begin{equation} \label{recipe for transfer} f_v^G(\delta) = \sum_{\gamma} \Delta_v(\delta, \gamma) f_{G(F_v)}(\gamma), \end{equation} 
where the sum is taken over representatives $\gamma$ of all the conjugacy classes of~$G(F_v)$. The factor~$\Delta(\delta, \gamma)$ is equal to~$1$ if~$\gamma \in \delta$ and to~$0$ otherwise.  This construction gives a map from $\mathscr{H}(G_v)$ to functions on stable conjugacy classes. Denote the image of this map by $\mS(G_v)$. A linear functional on $\mathscr{H}(G_v)$ is said to be \emph{stable} if it factors through $\mS(G_v)$. 

Now let $G_v$ be an arbitrary unitary group. For each endoscopic group $H_v$ of~$G_v$, the construction of transfer factors by Langlands-Shelstad \cite{LS87} and Kottwitz-Shelstad~\cite{KS99} gives rise to maps~$\mathscr{H}(G_v) \to \mS(H_v)$. The transfer factors are a significantly more delicate generalization of the $\Delta(\delta, \gamma)$ above; in particular, their normalization in \cite[\S 1.1.2]{KMSW} (and thus the notion of transfer) depends on the choice of pure inner form as in Remark \ref{remark about pure inner forms}. This provides a system of maps from the Hecke algebras to their stable counterparts, and two functions~$f_v \in \mathscr{H}(G_v)$ and~$f_v^{H_v} \in \mathscr{H}(H_v)$ will be said to form a \emph{transfer pair} if their images under their respective maps to~$\mS(H_v)$ agree. Although~$f_v^{H_v}$ is not uniquely determined by~$f_v$, we may abuse terminology and refer to a choice of~$f_v^{H_v}$ as the transfer of~$f_v$. 

To extend the notion of transfer to global test functions, it is first necessary to know that the transfer of characteristic functions of maximal compact subgroups of~$G_v$ are the corresponding functions on~$H_v$. This is the fundamental lemma, now a theorem due to Laumon-Ng\^o \cite{LN} in the case of unitary groups, and to Ng\^o \cite{Ngo} in general, after reductions by Waldspurger \cite{Wa06, Wa08}.
\begin{theorem}[Fundamental Lemma] \label{Fundamental Lemma}
	Let~$G_v$ and~$H_v$ be unramified reductive groups over a non-archimedean local field~$F_v$. Let $K(G_v)$ and $K(H_v)$ be respective choices of hyperspecial maximal compact subgroups. Then their characteristic functions $f_v = 1_{K(G_v)}$ and $f_v^{H_v} = 1_{K(H_v)}$ form a transfer pair. 
\end{theorem}
With this in mind, the transfer of a factorizable global test function $f = \prod_v f_v \in\Heckev$ is the product $f^{H} = \prod_{v} f_v^{H_v}$ of its transfers, a definition extended linearly to all of $\mathscr{H}(G)$. 
We will likewise define the global stable Hecke algebra~$\mS(G^*) := \ten'_v\mS(G^*_v)$. A linear functional on $\mathscr{H}(G^*)$ is \emph{stable} if it factors through $\mS(G^*)$. 
\subsubsection{Local Character Identities} \label{charid} The transfer of representations between $G$ and its endoscopic groups $H$ is encoded via identities between linear combinations of characters; the coefficients are determined by the pairings~\eqref{local pairing}.  We start with distributions $f^G(\psi)$ on $\mathscr{H}(G)$. Let $F$ be local and $G^*/F$ be a quasisplit unitary group or a product thereof, and $\psi$ be an Arthur parameter of $G^*$. Then Mok attaches a stable linear form to $\psi$.
\begin{theorem}[Theorem 3.2.1 (a), \cite{Mok}] Let $\psi \in \Psi(G^*)$. Then there exists a unique stable linear form $f \mapsto f^{G^*}(\psi)$ on $\mathscr{H}(G^*)$,  determined by transfer properties to $GL_N$. If $G^* = G^*_1 \times G^*_2$ and $\psi = \psi_1 \times \psi_2$, then $f^{G^*}(\psi) = f^{G^*_1}(\psi_1) \times f^{G^*_2}(\psi_2)$.
\end{theorem}

We will not discuss in detail the character identities relating $f^{G^*}(\psi)$ to traces on~$GL_N$, save for reminding the reader that this distribution is related to the trace~$\tr \pi_{\psi, N}(f)$ where~$\pi_{\psi, N}$ corresponds to~$\psi$ under the Local Langlands Correspondence. We will focus on the relation between the~$f^H(\psi^H)$ for the groups~$H \in \mE(G)$ and the characters of representations in~$\Pi_{\psi}$. If~$G = G^*$, these identities were established by Mok, and for inner forms by Kaletha-Minguez-Shin-White. Recall that~$s_\psi$ is the distinguished element of~$\mS_{\psi}$ defined in~\eqref{spsi}.  

\begin{theorem}[Theorem 3.2.1 (b), \cite{Mok}] \label{lci for qs}
	Let $G^*$ be a
	quasisplit unitary group, let~$\psi \in \Psi(G^*)$, and let~$\Pi_{\psi}$ be the associated Arthur packet 
	equipped with the pairing of equation~\eqref{local pairing}. Let $s_H \in \mS_{\psi}$ be such that $(H, \psi^H)$ correspond to $(\psi, s_H)$ under the correspondence of Lemma \ref{lemma on endoscopic bijection}. Then for a transfer pair $(f, f^H)$ we have \[ f^H(\psi^H) = \sum_{\pi \in \Pi_{\psi} }  \ip{s_\psi s_H}{\pi}  \tr \pi (f). \]
\end{theorem}

\begin{theorem}[Theorem 1.6.1, \cite{KMSW}] \label{local character identities for general unitary groups} Let $(G,\xi)$ be an inner form of $U(N)$ and let $\psi$, $\Pi_{\psi}$, $H$, $s_H$, and $(f, f^H)$ be as above. Let $e(G)$ be the Kottwitz sign. Then \[ f^H(\psi^H) = e(G) \sum_{\pi \in \Pi_{\psi}} \ip{s_\psi s_H}{\pi}  \tr \pi (f).   \] 
\end{theorem}

\begin{remark}\label{remark on conditionality}
Let us recall a discussion from the introduction: the proofs in~\cite{KMSW} are not given in full generality. For example, Theorem~\ref{local character identities for general unitary groups} is only proved for bounded parameters. The authors of \cite{KMSW} anticipate that the proof will appear in a pair of papers, the first of which \cite{KMS} should contain the results we use here. 
\end{remark}

\subsubsection{Local Packets for General Linear Groups} \label{Section on Local packets for GLN}

As discussed in \ref{section on quasisplit unitary groups}, if $F$ is local and corresponds to a place splitting in our global CM extension, then $G \simeq GL_N$.  In this situation the local Arthur packet and the pairing are especially simple. 

\begin{theorem}[Section 2, \cite{Mok}] \label{lemam on packets for $GL_N$}
	If $G = GL_N$ and $\psi$ is an Arthur parameter for $G$, then the packet $\Pi_{\psi}$ contains one element: the irreducible representation associated to $\phi_{\psi}$ by the local Langlands correspondence. The character $\ip{\;}{\pi_{\psi}}$ is trivial. 
\end{theorem}

We now consider character of identities between representations of $G$ and those of its endoscopic groups. They are alluded to in \cite{Mok} and \cite{KMSW}, but we give a more explicit description based on \cite[\S 3.3]{Sh11}. For~$G=GL_N$, stable and regular conjugacy classes coincide, so~$\mS(G) = \mathscr{H}(G)$. Since the global extension giving rise to our unitary group is~CM, we  may assume that~$F$ is non-archimedean. If~$H = GL_{N_1}\times GL_{N_2}$ with $N_1 + N_2 = N$, then the embedding $\embend$ realizes~$H$ as a Levi subgroup of~$G$. Let~$P = HN$ be a parabolic subgroup of $G$ containing~$H$. Given~$f \in \mathscr{H}(G)$, define the constant term along $P$ as \[  f^{P}(h) := \del^{1/2}_{P}(h) \int_{N} \int_{K}f(khnk^{-1})dkdn, \quad h \in H(F). \] Here the integrals are taken with respect to suitably normalized Haar measures and~$\delta_{P}$ is the modulus character. The function~$f_v^P$ is smooth and compactly supported, and by results of van Dijk~\cite{va72}, it satisfies the requisite orbital integrals identities to be a transfer of~$f$, so we let~$f^H := f^P$. If~$f$ is unramified, then~$f^H$ is the image of~$f$ under the map~$\mathscr{H}(G)^{ur} \to \mathscr{H}(H)^{ur}$ induced by the Satake isomorphism. Thus this notion of transfer satisfies the fundamental lemma. 

For a parameter~$\psi$ of $G$, we let~$f^{G}(\psi) = \tr \pi_{\psi}(f)$ \cite[\S 1.5]{KMSW} for the unique~$\pi_{\psi} \in \Pi_{\psi}$ and extend this definition multiplicatively to products of general linear groups. 
Let~$\pi_{\psi}^H$ be the unique representation in the packet associated to~$\psi^H$. It follows from the local Langlands correspondence (see for example~\cite[p.6]{HT01} and note that the twist therein is accounted for here in the definition of the embedding ~$\embend$) that ~$\pi_{\psi} = \mI_P(\pi^H_{\psi})$, where~$\mI_P$ denotes normalized parabolic induction with respect to~$P$. In view of this and of Theorem~\ref{lemam on packets for $GL_N$}, the local character identities for~$GL_N$ amount to an equality of traces between~$\tr\pi(f^H)$ and the trace of $f$ on the corresponding induced representation. Again this is a result of van Dijk, which we record below.
\begin{theorem}[Section 5, \cite{va72}]\label{lci for GLN}
	Let $G, H, P$, and $f^H$ be as above. Let~$\pi$ be a unitary irreducible representation of~$H$ and let~$\mI_P(\pi)$ be its normalized parabolic induction with respect to~$P$. Then~$\tr \pi (f^H) = \tr \mI_P(\pi)(f)$. 
\end{theorem}

\subsection{The Trace Formula and its Stabilization} \label{section on the trace formula} 
We now introduce Arthur's trace formula following~\cite[\S 3]{KMSW} (see~\cite[\S 3]{A} for a more detailed exposition), focusing on the statements needed for our applications. The rough picture is as follows: for~$F$ a number field, and a connected reductive group~$G/F$, the trace formula~$\Idisc$ (sometimes denoted $I^G_{\text{disc}}$) is a distribution on the Hecke algebra~$\mathscr{H}(G)$, defined in terms of the traces of intertwining operators on variants of~$\Aut(G(F) \dom G(\BAF))$ indexed by a system of Levi subgroups of~$G$. The contribution of the group $G$ itself is the trace $\tr \Rdisc(f) := \tr R(f)$ introduced in Section \ref{SectionOnTestFunctions}. The trace formula admits two decompositions: a spectral one into a sum over the contributions of~$\psi \in \PsiG$, and an endoscopic one (or stabilization) into a sum of stable distributions on endoscopic groups. Our proof will follow from the interplay of these two decompositions.
\subsubsection{Contribution of a Parameter} We start by directly introducing the distributions given by the contribution of each Arthur parameter as in~\cite[\S 3.3]{KMSW}. For following paragraphs, let $(G, \bcend)$ be a pair consisting an inner form of a (possible product of) unitary groups, and an embedding~$\bcend: {^L}G \to {^L}\GN$ as in Section~\ref{section on morphisms of L-groups}. When~$G$ is an inner form of $U(N)$, we have $\bcend = \emb$. 

Recall $\mC(G)$, the set of families of conjugacy classes introduced in~\ref{parameters determine conjugacy classes}. To any automorphic representation $\pi$ of $G$, we can associate an element $c_\pi \in \mC(G)$ by letting $c_{\pi,v}$ be the Satake parameter of $\pi_v$ at all the unramified places $v$. Likewise, we associate to $\pi$ an infinitesimal character $\mu_\pi$. Then for $c \in \mC(G)$ and a positive real number $t$, the distribution $\Idisctc$ is described in~\cite[\S 3.1]{KMSW}. It is the restriction of the traces defining $\Idisc$ to representations $\pi$ such that $c = c_\pi$, and such that~$\mu_\pi$ satisfies~$|\text{Im}\mu_\pi| = t$ under a suitable metric. To go from conjugacy classes to parameters, recall that in~\ref{parameters determine conjugacy classes} we identified ~$\Psi(N)$ with~$\mC_{aut}(N) \subset \mC(N)$. To each~$\psi^N \in \Psi(N)$ is thus associated an element~$c(\psi^N) \in \mC_{aut}(N)$ as well as a positive real number $t(\psi^N)$ coming from the infinitesimal character of $\psi^N$. For each 
parameter~$\psi^N \in \Psi(N)$, we follow \cite[\S 3.3]{KMSW} and define
\[ I_{\text{disc}, \psi^N, \bcend}=  \sum_{c \mapsto c(\psi^N), t \mapsto t(\psi^N)} \Idisctc. \] The sum runs over the~$c \in \mC(G)$ that map to~$c(\psi^N)$ under the map~$\mC(G) \to \mC(N)$ induced by~$\bcend$. When $G^* = U(N)$, we follow \cite[\S3.3]{KMSW} and shorten~$I_{\text{disc}, \psi^N, \eta_{\kappa}}$ to~$\Idisci$ when $\psi = (\psi^N, \tilde{\psi}) \in \Psi(G^*, \eta_{\kappa})$, using the injection $\eta_{\kappa,*}$ of Section~\ref{SectionOnGlobalArthurParameter}. We similarly obtain distributions~$ \tr R_{\text{disc},c,t}$, $R_{\text{disc},\psi^N, \eta_{\kappa, \underline{\kappa}}}$ and~$ \tr R_{\text{disc},\psi^N,\eta_{\kappa}} :=\tr \Rdisci$. If we have $G^* \in \tilde{\mE}_{sim}(N)$, as well as  $(H,\embend) \in \mE(G)$ and $\psi \in \Psiqs$, we will also shorten notation and denote $I^H_{\text{disc},\psi} = I^H_{\text{disc},\psi^N, \embend \circ \eta_{\kappa}}$.

An essential step in the proof of the endoscopic classification of representations is showing that~$\tr \Rdisci$ computes the traces of the representations in~$\Pi_{\psi}$, provided that~$\psi$ satisfies the two conditions of Remark~\ref{our restriction on the parameter}. 
\begin{theorem}[From \cite{Mok}, (5.7.27), and \cite{KMSW}, proof of Theorem 5.0.5] \label{Rdisc computes the traces on the packet.}
Let~$\psi \in \PsiG$ be a square-integrable parameter associated to~$\Pi_{\psi}$, and let~$f \in \mathscr{H}(G)$. Then \[ \tr \Rdisci(f) = \sum_{\pi \in \Pi_{\psi}} m(\pi) \tr \pi(f).\] 
	
\end{theorem}

In the notation of Section 2.5, the multiplicity $m(\pi)$ is equal to $1$ if~$\ip{\pi}{\;}_\psi = \epsilon_\psi$ as characters of $\mS_\psi$, and $0$ otherwise, see \cite[\S 1.7]{KMSW}. Note again that~\cite[Theorem 5.0.5]{KMSW} is stated, but not fully proved, in the case of non-generic parameters, as mentioned in the introduction and in Remark~\ref{remark on conditionality}. 

Following a result of Bergeron-Clozel, the distributions~$\tr \Rdisci$ and~$\Idisci$ agree if the infinitesimal character is regular.

\begin{theorem}[Theorem 6.2, \cite{BC}] \label{Bergeron-Clozel on parameters with reg inf char}
Let~$G$ be a connected reductive group. Let~$\psi \in \Psi(G)$ be a global Arthur parameter such that~$\psi_\infty$ has regular infinitesimal character. Then the contributions of the Levi subgroups~$M \neq G$ to the distribution~$\Idisci$ vanish. In particular for all~$f \in \mathscr{H}(G)$ we have~$\Idisci(f) =\tr \Rdisci(f)$.
\end{theorem}
If $H = U(N_1) \times U(N_2)$, and $\psi^H = \psi_1 \times \psi_2$ we can write \[ R_{\text{disc}, \psi^H}(f) = R_{\text{disc},\psi_1}(f_1)\cdot R_{\text{disc},\psi_2}(f_2), \quad f = f_1 \times f_2 \in \mathscr{H}(H). \] Following the above result, we can also write this as $I_{\text{disc}, \psi^H}(f)$ provided that $\psi^H$ has regular infinitesimal character. 

\subsubsection{Stabilization} \label{Stabilization}
We now recall the identity that drives our theorems: the stabilization of~$\Idisci$, i.e. its decomposition into sum of stable traces of the transfers~$f^H$ of~$f$ for the endoscopic groups~$H \in \mE(G)$.  Our references are to Arthur~\cite{A}, but the versions for unitary groups are formally identical, see for example~\cite[(3.3.2)]{KMSW} and~\cite[(4.2.1)]{Mok}. Recall that~$\tilde{\Psi}(N)$ is the set of conjugate self-dual parameters, and~$\PsiG \subset \tilde{\Psi}(N)$.

\begin{theorem}[\cite{A}, Corollary 3.3.2(b)] Suppose that~$\psi \in \tilde{\Psi}(N)$ and let~$f \in \mathscr{H}(G)$. Then for each endoscopic datum~$(H, \xi_H) \in \mE(G)$ there is a constant~$\iota(G,H)$ and stable distributions~$\SHdisci$ on~$\mathscr{H}(H)$, defined inductively, such that \begin{equation} \label{stable trace formula} \Idisci (f) = \sum_{H \in \mathcal{E}(G)}
		\iota(G,H) \SHdisci(f^H). \end{equation}
\end{theorem}

\begin{remark}
	For unitary groups, the global factor~$\iota(G,H)$ is introduced in \cite[\S 4.2]{Mok} and \cite[\S 3.1]{KMSW}. It is independent of the inner form~$G$. If~$G = U(N)$ and~$H = U(N_1) \times U(N_2)$, then following~\cite[4.2]{Mok} we have \begin{equation}\label{bounds on side of i(G,H)} \iota(G,H) = \begin{cases}
			1 & N_1N_2 = 0 \\ \frac{1}{2} & N_1, N_2 \neq 0, N_1 \neq N_2 \\ \frac{1}{4} & N_1 = N_2 \neq 0. 
	\end{cases} \end{equation} 
\end{remark}

\section{Upper Bounds from the Stabilization} \label{section on upper bounds}

In this section we unpack the summands of the stabilization of $\Idisci$ and extract upper bounds on the trace of test functions from the character identities. 

\subsection{The Stable Multiplicity Formula}

Let $\psi = (\psi^N, \tilde{\psi}) \in \Psi(G^*, \eta_{\kappa})$. Recall the decomposition from~\eqref{stable trace formula}: \begin{equation} \label{stable again}\Idisci (f) = \sum_{H \in \mathcal{E}(G)}
	\iota(G,H) \SHdisci(f^H).\end{equation} The stable multiplicity formula expresses each~$\SHdisci$ as a sum of traces. If~$f_v$ is a local test function and $\psi_v$ a local parameter, the formula for~$f^{H_v}(\psi_v)$ was given in Section~\ref{charid}. If~$f = \prod_v f_v$ and $\psi$ are global, we write~$f^H(\psi) := \prod_v f^{H_v}(\psi_v)$. The group~$\mS_{\psi}$ and the element~$s_\psi$ were defined in~\ref{section where we compute the stabilizers}, and~$\epsilon_\psi$ in~\ref{sectionon epsilon factors}. The stable multiplicity formula, only defined for quasisplit groups, is the following expression: 
\begin{theorem}[\cite{Mok}, Theorem 5.1.2]
	For~$\psi \in \Psi(G, \eta_{\kappa})$, we have \[ S^G_{\text{disc}, \psi}(f) = |\mS_\psi|^{-1}\epsilon_\psi^G(s_\psi)\sigma(\bar{S}_\psi^0)f^G(\psi). \]
\end{theorem}
For any connected reductive group~$S$, the quantity~$\sigma(S)$  was defined by Arthur in~\cite[\S 4.1]{A}. The centralizers $S_\psi$ of our~$\psi$ are always finite, so~$\bar{S}^0_\psi$ is trivial and~$\sigma(\bar{S}^0_\psi) = 1$, see~\cite[Remark 5.1.4]{Mok}. The stable multiplicity formula is stated for~$G$ a unitary group (in which case the map~$\psi \mapsto \psi^N$ is injective), but can be extended to products~$H  \in \mE(G)$. Let~$\Psi(H, \psi^N)$ be the set consisting of parameters~$\psi^H = (\psi^{N,H}, \tilde{\psi}^H)$ with~$\psi^{N,H} = \psi^N$. The stable multiplicity formula for~$H$, given in \cite[(5.6.3)]{Mok}, is: \begin{equation} \label{stable multiplicity formula for endoscopic groups} S^H_{\text{disc}, \psi}(f^H) = \sum_{\psi^H \in \Psi(H, \psi^N)} \frac{1}{|\mS_{\psi^H}|} \epsilon_\psi^H(s_{\psi^H}^H)\sigma(\bar{S}_{\psi^H}^0)f^H(\psi^H). \end{equation} 
We combine equations~\eqref{stable again} and~\eqref{stable multiplicity formula for endoscopic groups} and rewrite the resulting expression as a sum over pairs~$(H, \psi^H)$ to get: \begin{equation} \label{first simplification of the STF}\Idisci (f) = \sum_{(H, \psi^H)} \iota(G,H) \frac{1}{|\mS_{\psi^H}|} \epsilon_\psi^H(s_{\psi^H}^H)\sigma(\bar{S}_{\psi^H}^0)f^H(\psi^H).\end{equation}
We now collect the terms that can be bounded uniformly, and let \begin{equation}C(\psi,H) := \iota(G,H) \sigma(\bar{S}_{\psi^H}^0)|\mS_{\psi^H}|^{-1}. \end{equation}

\begin{lemma} \label{lemma with the constant} Let~$\psi \in \Psiqs$ and let~$(H, \xi_H, s_H) \in \mE(G)$ be an endoscopic datum such that~$\psi$ factors through~$\xi_H$. Then 
	\begin{itemize}
		\item[(i)] The contribution of~$(H, \psi^H)$ to the sum \eqref{first simplification of the STF} is equal to~$ C(\psi, H) \epsilon_\psi^H(s_\psi^H)f^H(\psi^H).$
		\item[(ii)] The constant ~$C(\psi, H)$ is bounded uniformly in $\psi$ and $H$: it always satisfies $2^{-(N+1)}\leq C(\psi,H) \leq 1$, where $N$ is the rank of $G$. 
	\end{itemize}
\end{lemma}
\begin{proof}
	Part (i) follows immediately from~\eqref{first simplification of the STF} and it suffices to exhibit the bound on~$ C(\psi, H)$. As stated above, we have $\sigma(\bar{S}_\psi^0) =1$ since $\psi$ is elliptic. We also gave uniform bounds on $\iota(G,H)$ in \eqref{bounds on side of i(G,H)} and on $|\mS_{\psi}|$ in \eqref{size of S_psi}.
\end{proof}
In Lemma~\ref{lemma on endoscopic bijection}, we gave a bijection between~$\mS_{\psi}$ and the set of pairs~$(H, \psi^H)$. We use it to re-index the sum~\eqref{first simplification of the STF} and obtain the expression \begin{equation}\label{simple form in terms of Spsi}\Idisci(f) = \sum_{s_H \in \mS_{\psi}} C(\psi, s_H) \epsilon_\psi^H(s_{\psi^H}^H)f^H(\psi^H). \end{equation}
This sum depends on parameters and representations of~$H$, which we want to rewrite in terms of~$G$. For~$\epsilon_\psi$, we use Mok's so-called \emph{endoscopic sign lemma}. 
\begin{lemma}[Lemma 5.6.1, \cite{Mok}] \label{endoscopic sign lemma}
	Let~$(H, \xi, s_H) \in \mE(G)$ and~$\psi \in \Psiqs$ be such that~$(H, \psi^H)$ corresponds to~$(\psi, s_H)$. Let~$\epsilon_\psi^{G^*}$ and~$\epsilon_{\psi}^H$ be the respective characters of~$\psi$ and~$\psi^H$. Let~$s_{\psi^H}^H$ be the image of~$\psi^H(-I)$ in the quotient $\mS_{\psi}^H$ associated to $H$. Then we have  \[ \epsilon_\psi^{H}(s^H_{\psi^H}) = \epsilon^{G^*}(s_\psi s_H).  \] 
\end{lemma}
We can now rewrite~$\Idisci(f)$ in a form conducive to extracting bounds.
\begin{prop}\label{trace over SH}
	Let $\psi \in \PsiG$, and let $f \in \mathscr{H}(G)$ be factorizable. Then  \begin{align} \Idisci(f) &= \sum_{s_H \in \mS_{\psi}} C(\psi, s_H)\epsilon^{G^*}_\psi(s_\psi s_H) \prod_v\left(\sum_{\pi_v \in \Pi_{\psi_v}} \ip{s_{\psi_v} s_{H_v}}{\pi_v}  \tr \pi_v(f_v)\right)  \nonumber \\
		&\label{intermediate version} = \sum_{s_H \in \mS_\psi} C(\psi,s_H) \sum_{\pi \in \Pi_{\psi}} \epsilon^{G^*}_\psi(s_\psi s_H) \ip{s_\psi s_H}{\pi} \tr \pi (f). 
	\end{align}
\end{prop}
\begin{proof}
	We start from the equality \eqref{simple form in terms of Spsi}: \[ \Idisci(f) = \sum_{s_H \in \mS_{\psi}} C(\psi, s_H) \epsilon_\psi^H(s_\psi^H)f^H(\psi^H).\] The distribution $f^H(\psi^H)$ was defined as $f^H(\psi^H) = \prod_v f_v^{H_v}(\psi_v^H)$. Each local factor can be written in terms of the trace of representations in $\Pi_{\psi_v}$ by Theorems \ref{lci for qs}, \ref{local character identities for general unitary groups}, and ~\ref{lci for GLN}. In all cases, the identity is: \[ f_v^{H_v}(\psi_v^{H_v}) = e(G_v) \sum_{\pi_v \in \Pi_{\psi_v}} \ip{s_{\psi_v} s_{H_v}}{\pi_v}  \tr \pi_v (f_v). \] The local Kottwitz signs cancel out globally, and using Lemma \ref{endoscopic sign lemma}, we rewrite \[ \Idisci(f) = \sum_{s_H \in \mS_{\psi}} C(\psi, s_H)\epsilon^{G^*}_\psi(s_\psi s_H) \prod_v \left(\sum_{\pi_v \in \Pi_{\psi_v}} \ip{s_{\psi_v} s_{H_v}}{\pi_v}  \tr \pi_v(f_v)\right). \]  At all but finitely many~$v$, we have~$f_v = 1_{K_v}$ for a hyperspecial maximal compact subgroup~$K_v$. At these places,~$\tr \pi_v (f_v)$ is only nonzero on~$K_v$-unramified representations~$\pi_v$. Unramified local packets contain exactly one unramified representation following~\cite[Proposition 1.5.2 (5)]{KMSW} so we interchange the sum and product to get \[ \prod_v \left(\sum_{\pi_v \in \Pi_{\psi_v}} \ip{s_{\psi_v} s_{H_v}}{\pi_v}  \tr \pi_v(f_v)\right) = \sum_{\pi \in \Pi_{\psi}} \left(\prod_v\ip{s_{\psi_v} s_{H_v}}{\pi_v} \right) \tr \pi (f)  \] for~$\pi = \ten_v \pi_v$. Using the definition~$\ip{\cdot}{\pi} := \prod_v\ip{\cdot}{\pi_v}$, we rewrite \[\Idisci(f) = \sum_{s_H \in \mS_\psi} C(\psi,s_H) \sum_{\pi \in \Pi_{\psi}} \epsilon^{G^*}_\psi(s_\psi s_H) \ip{s_\psi s_H}{\pi} \tr \pi (f). \]
\end{proof}

\subsection{Upper Bounds and the Dominant Group} \label{Bounds by the Dominang Group}

Recall once more the bijection~$ (H,\psi^H) \leftrightarrow (\psi, s_H)$ from \ref{section on endoscopic data and parameters}. We will single out one object on either side, and show that for certain~$f$, its contribution to the distribution~$\Idisci(f)$ bounds the others. Recall that $s_\psi \in \mS_{\psi}$ was the image of the matrix~$-I \in SL_2$ under~$\psi$. 

\begin{definition}
	Let $(H_\psi, \psi^{H_\psi})$ be the pair corresponding to the pair $(\psi, s_\psi)$ containing the distinguished element $s_\psi$ under the bijection $ (H,\psi^H) \leftrightarrow (\psi, s_H)$.
\end{definition}
Note that it is possible that $H_\psi = G$, for example when $\psi$ is bounded. \begin{definition}\label{defstableshorthand}
	Let~$\psi \in \PsiG$ and let~$(H,\xi_H,s_H)$ be such that~$\psi$ factors through~$\xi_H$. Let~$f$ be a global test function. Then define \[ S(\psi,s_H, f) = C(\psi,s_H) \sum_{\pi \in \Pi_{\psi}} \epsilon^{G^*}_\psi(s_\psi s_H) \ip{s_\psi s_H}{\pi} \tr \pi (f). \]
\end{definition}
Proposition \ref{trace over SH} can then be reformulated as stating that: 
\begin{equation}\label{eqTF simplified}
	\Idisci(f) = \sum_{s_H \in \mS_\psi} S(\psi,s_H,f). 
\end{equation} 
\begin{lemma} \label{lemma main coincidence}
	If $H = H_\psi$, then \begin{equation}
		S(\psi,s_\psi,f) = C(\psi,s_\psi)\sum_{\pi \in \Pi_\psi} \tr(\pi)(f).
	\end{equation} 
\end{lemma}
\begin{proof}
	This follows since $s_{H_\psi} = s_\psi$ by definition. Since $\mS_\psi \simeq (\BZ/2\BZ)^{r-1}$, this implies that $\epsilon_\psi^{G^*}(s_\psi^2) = 1$ and $\ip{s_\psi^2}{\pi}= 1$ for all $\pi$.
\end{proof}

This allows us to state our main application of the stable trace formula.
\begin{theorem} \label{main theorem}
	Let $G$ be a unitary group, let~$\psi \in \PsiG$, and let~$f \in \mathscr{H}(G)$ be a factorizable test function with~$\tr \pi(f)$ real and nonnegative for all~$\pi \in \Pi_\psi$. Then there exist a constant~$C(\psi)$ such that \[ \Idisci(f) \leq C(\psi) S(\psi,s_\psi, f). \] The constant $C(\psi)$ satisfies $2^{-(N+1)} \leq C(\psi) \leq 2^{2N}$; it is thus bounded above and below independently of $\psi$. 
\end{theorem}
\begin{proof}
	We compare the various terms appearing in~\eqref{eqTF simplified}: 	\[	\Idisci(f) = \sum_{s_H \in \mS_\psi} S(\psi,s_H,f).  \] Ignoring for a moment the constants $C(\psi, s_H)$, the summands only differ from one another via the signs~$\epsilon^{G^*}_\psi(s_\psi s_H) \ip{s_\psi s_H}{\pi} \in \{\pm 1\}$ appearing as coefficients of the traces~$\tr \pi(f)$. In the term coming from~$s_\psi$, we get \[ S(\psi, s_\psi,f) = C(\psi, s_\psi) \sum_{\pi \in \Pi_\psi} \tr(\pi)(f) \] from Lemma \ref{lemma main coincidence}. For any other~$s_H \in \mS_{\psi}$, the coefficients~$\epsilon^{G^*}_\psi(s_\psi s_H) \ip{s_\psi s_H}{\pi}$ have the potential to be equal to~$-1$. Thus if~$\tr \pi(f) \geq 0$ for all~$ \pi \in \Pi_{\psi}$, we have \begin{align*}
		S(\psi, s_H, f)  &= C(\psi,s_H) \sum_{\pi \in \Pi_{\psi}} \epsilon^{G^*}_\psi(s_\psi s_H) \ip{s_\psi s_H}{\pi} \tr \pi (f) \\ &\leq C(\psi, s_H) \sum_{\pi \in \Pi_\psi} \tr(\pi)(f) = 	\frac{C(\psi, s_H)}{C(\psi, s_\psi)} \cdot S(\psi, s_\psi, f). 
	\end{align*} Summing over the~$s_H$ we get \[ \Idisci(f) \leq \left(\frac{\sum_{s_H \in \mS_{\psi}}{C(\psi, s_H)}}{C(\psi, s_\psi)}\right) S(\psi,s_\psi, f): = C(\psi)S(\psi,s_\psi, f) . \] For the bounds, we showed in Lemma \ref{lemma with the constant} that $2^{-(N+1)} \leq C(\psi, s_H) \leq 1$. As for the cardinality of~$\mS_{\psi}$, it is bounded between $1$ and $2^{N-1}$ as we saw in Section \ref{section where we compute the stabilizers}. 
\end{proof}
In practice, the group $H_\psi$ is easily computed from $\psi \mid_{SL_2}$. \begin{lemma}
	Let $\psi = \boxplus_i (\mu_i \boxtimes \nu(m_i)) \in \Psi(N)$ be a global square-integrable Arthur parameter, and let $N_1 = \sum_{m_i \equiv 1 \mathrm{ mod }\, 2} m_i$. Then the group $H_\psi$ is \[ H_\psi  = U(N_1) \times U(N-N_1). \] 
\end{lemma}
\begin{proof}
	By definition $ s_\psi= \psi(1,-I) \in GL_N$. The image of~$-I$ under the~$m$-dimensional representation of $SL_2$ is $(-1)^{m+1}I_m$. Thus~$s_\psi = \text{diag}(-I_{N_1}, I_{N_2})$, where $N_1 = \sum_{m_i \equiv 1 \mathrm{ mod }\, 2} m_i$ and $N_2 = N-N_1$, with centralizer $GL_{N_1 } \times GL_{N_2}$. 
\end{proof}
	The image $\psi(SL_2)$ and the group $H_\psi$ are determined by any localization $\psi_v(SL_2)$.  In Section \ref{section on cohomology}, we will use this, together with the known (archimedean) parameters of cohomological representations, to bound growth of cohomology.

\subsection{Hyperendoscopy} 
\label{section on hyperendoscopy} We recall the notion of hyperendoscopic datum first introduced by Ferrari \cite{Fe}. We will use it to bound the expression $S(\psi, s_\psi, f)$. As pointed out by Dalal \cite{Dalal20Sato}, the results of \cite{Fe} do not quite hold in full generality, but they do hold for unitary groups, which have simply connected derived subgroups.

\begin{definition}
	A chain of hyperendoscopic data for the (local or global) group~$G$ is a collection \[ \mH = (G, H_1,...,H_q), \] where~$H_1$ is a proper endoscopic datum for $G$ and $H_{i+1}$ is a proper endoscopic datum for $H_i$. 
\end{definition} 
The integer $p(\mH) = q$ is the \emph{depth} of the datum. Denote  \[ \iota(\mH) = (-1)^{p(\mH)}\iota(G,H_1)\cdot \iota(H_1,H_2)\cdot ... \cdot \iota(H_{q-1},H_q), \] and $\Idisc^{\mH} := \Idisc^{H_q}$. As with endoscopic data, two chains of hyperendoscopic data will be considered equivalent if they are conjugate under $\hat{G}$. Ferrari denotes by~$\mH\mE(G)$ the collection of equivalence classes of chains of hyperendoscopic data for~$G$.  If~$\psi \in \Psi(G)$ is an Arthur parameter, we will denote by $\mH\mE(G,\psi)$ the collection of equivalence classes of chains of hyperendoscopic data $\mH \in \mH\mE(G)$ such that~$\psi$ factors through~ the embedding $\xi_{p(\mH)}$ associated to $H_{p(\mH)}$. Note that the depth of chains in $\mH\mE(G,\psi)$ is bounded above by the number of simple constituents of $\psi$.

If $\mH \in \mH\mE(G)$ is a chain of hyperendoscopic data, and~$f^G$ is a test function, we inductively define~$f^{H_{i+1}} = (f^{H_i})^{H_{i+1}}$. 
The function $f^{H_{p(\mH)}}$ depends on a choice of transfer $f^{H_i}$ at each step. We allow this, but require that our choice of $f^{H_i}$ be consistent over chains that are truncations of one another. 
The following is the specialization to a parameter~$\psi$ of a trick initially discovered by Ferrari \cite[3.4.2]{Fe}. 
\begin{proposition}\label{prop hyperendoscopy} Let~$(G, \eta) \in \tilde{\mE}(N)$ be quasisplit and let~$\psi^N \in \Psi(N)$. Then \[ S^G_{\text{disc}, \psi^N, \eta}(f) = \sum_{\mH \in \mH\mE(G,\psi)} \iota(\mH) I_{\text{disc},\psi^N, \eta \circ \xi_{p(\mH)}}^{H_q} (f^{H_q}) \]
\end{proposition}
\begin{proof} If $\psi^N \in \Psi_{\text{sim}}(N)$, then  $I^G_{\text{disc}, \psi^N, \eta}(f) = S^G_{\text{disc}, \psi^N, \eta}(f)$ and the result holds trivially. Otherwise, we have that \begin{equation} \label{eq first hyperendoscopy} S^G_{\text{disc}, \psi^N, \eta}(f) = I^G_{\text{disc}, \psi^N, \eta}(f) - \sum_{H \in \mE(G, \psi^N)} \iota(G,H)S^H_{\text{disc}, \psi^N, \eta \circ \xi}(f^H).  \end{equation} By induction, for each $H$ in $\mE(G,\psi)$, we have \begin{equation} \label{eqn induction hypothesis hyperendoscopy} S^H_{\text{disc}, \psi^N, \eta \circ \xi}(f^H) =\sum_{\mH \in \mH\mE(G,\psi)} \iota(\mH) I_{\text{disc},\psi^N, \eta \circ \xi_{p(\mH)}}^{H_q} (f^{H_q}).  \end{equation} By construction, each~$\mH \in \mH\mE(G, \psi^N)$ is obtained from a hyperendoscopic datum~$\mH' \in \mH\mE(H, \psi^N)$ for some~$H \in \mE(G, \psi^N)$, and~$p(\mH) = p(\mH') +1$. Substituting~\eqref{eqn induction hypothesis hyperendoscopy} into~\eqref{eq first hyperendoscopy} yields the result.

\end{proof}
Recall that when~$G \in \tilde{\mE}_{sim}(N)$, the map~$\psi \mapsto \psi^N$ is injective. On the other hand, if~$H$ is an product of unitary groups, there could be several parameters~$\psi^H$ for~$H$ such that~$\psi^H \mapsto \psi^N$ under $\bcend$. From~\cite[\S 5.6]{Mok} we see that if~$H = H_1 \times H_2$ with~$H_i = U(N_i)$, and~$f^H = f^{H_1} \times f^{H_2}$, then \[ S^H_{\text{disc},\psi^N, \bcend}(f^H) = \sum_{\psi^H = \psi_1 \times \psi_2, \; \psi^H \mapsto \psi^N} S^{H_1}_{\mathrm{disc}, \psi_1}(f^{H_1}) \times S^{H_2}_{\mathrm{disc}, \psi_2}(f^{H_2}).\] The expression~$S(\psi,s_H,f)$ of Definition~\ref{defstableshorthand} picks out one of these summands. \begin{definition}
Let~$H = H_1 \times H_2$ as above, and let~$\mH \in \mH\mE(H)$ with~$p(\mH) = q$. Then~$H_q = H_{q_1} \times H_{q_2}$ with~$H_{q_i} \in \mH\mE(H_i)$. Let~$\psi^H = \psi_1 \times \psi_2 \in \Psi(H,\bcend)$. For a test function~$f^{H_q} = f^{H_{q_1}} \times f^{H_{q_2}}$, define \[ I^{H_q}_{\text{disc}, \psi^H}(f^H) = I^{H_{q_1}}_{\text{disc}, \psi_1}(f^{H_{q_1}}) \times I^{H_{q_2}}_{\text{disc}, \psi_2}(f^{H_{q_2}}). \] 
\end{definition} 
Here we use the notation $I^{H_{q_i}}_{\text{disc},\psi_i}$ as in Section \ref{section on the trace formula} since $U(N_i) \in \tilde{\mE}_{sim}(N_i)$. 
\begin{corollary} Let~$H = H_1 \times H_2$ as above and let~$\psi^H = \psi_1 \times \psi_2 \in \Psi(H, \embend)$ so that~$(H, \psi^H)$ corresponds to~$(\psi^N,s_H)$ under the correspondence of Lemma~\ref{lemma on endoscopic bijection}. Assume that $f^H = f^{H_1} \times f^{H_2}$. Then 
	\[S(\psi,s_H,f) = \iota(G,H) \sum_{\mH \in \mH\mE(H,\psi)} \iota(\mH)  I^{H_q}_{\text{disc},\psi^H}(f^{H_q}). \] \label{corollary on hyperendoscopy S(psi,sH,f)}
\end{corollary}
\begin{proof}
	We see in~\cite[\S 5.6]{Mok} that the term $ S^{H_1}_{\mathrm{disc}, \psi_1}(f^{H_1}) \times S^{H_2}_{\mathrm{disc}, \psi_2}(f^{H_2})$ is equal to \[ \frac{1}{|\mS_{\psi_1}||\mS_{\psi_1}|}\epsilon^{H_1}(\psi_1)\epsilon^{H_2}(\psi_2)f^{H_1}(\psi_1)f^{H_2}(\psi_2) = \frac{1}{|\mS_{\psi}|}\epsilon^H(\psi)f^H(\psi). \] By the argument of Proposition~\ref{trace over SH}, the last expression is equal to \begin{align*} \frac{1}{|\mS_{\psi}|}\sum_{\pi \in \Pi_{\psi}} \epsilon^{G^*}_{\psi}(s_{\psi} s_H) \ip{s_{\psi} s_H}{\pi} \tr \pi (f) &= \frac{S(\psi, s_H, f) \cdot }{|\mS_{\psi}|C(\psi, s_H)} 
	=\frac{S(\psi, s_H, f)}{\iota(G,H)}. \end{align*}
	Applying Proposition \ref{prop hyperendoscopy} to each factor of  $S^{H_1}_{\mathrm{disc}, \psi_1}(f^{H_1}) \times S^{H_2}_{\mathrm{disc}, \psi_2}(f^{H_2})$, we get \begin{align*}  \frac{S(\psi, s_H, f)}{\iota(G,H)} &= \left(  \sum_{\mH_1 \in \mH\mE(H_1, \psi_1)} \iota(\mH_1)  I^{H_{q_1}}_{\text{disc}, \psi_1}(f^{H_{q_1}}) \right) \cdot \left(  \sum_{\mH_2 \in \mH\mE(H_2, \psi_2)} \iota(\mH_2)  I^{H_{q_2}}_{\text{disc}, \psi_2}(f^{H_{q_2}}) \right) \\ 
& = \sum_{\mH \in \mH\mE(H, \psi)} \iota(\mH)  I^{H_q}_{\text{disc},\psi^H}(f^{H_q}) \end{align*}
\end{proof}

\section{Limit Multiplicity} \label{section on limit multiplicity}
Here we apply the results of the previous section to the limit multiplicity problem. 

\subsection{Level Structures} \label{section on level structures}
Let~$\mO_E$ and~$\mO_F$ be the rings of integers of the global fields~$E$ and~$F$. We introduce sets of places of~$F$: 
\begin{itemize}
	\item[-]~$S_f$ is a finite set of finite places of $F$, containing the places 
	which ramify in~$E$ as well as the places below those where the character~$\chi_{-}$ introduced in Section~\ref{section on notation} is ramified. 
	\item[-] $S_\infty$ is the set of all infinite places of~$F$. 
	\item[-] $S_0 \subsetneq S_\infty$ is a nonempty subset of the infinite places.
	\item[-] $S = S_f \cup S_\infty$.
\end{itemize} 
Note that the third requirement implies that~$F \neq \BQ$. Let~$\cp$ be an ideal of~$F$ with residue characteristic strictly greater than~$N^2[F:\BQ]+1$, corresponding to a place~$v_\cp \notin S$. For each finite place $v$ of $F$, denote by $\MO_{F_v}$ the ring of integers of $F_v$, and let $\hat{\MO}_F = \prod_v \MO_{F_v}$, and similarly for $\hat{\MO}_E$. We define the subgroups~$U(N,\cp^n) \subset U(N, \BAFf)$ to be \[ U(N,\cp^n) := \{ g \in U(N, \hat\MO_F) \subset GL_N(\hat\MO_E) \mid g \equiv I_N \; (\cp^n \MO_E) \}. \] For any finite place~$v$ of~$F$, let~$U(N,\cp^n)_v = U(N,\cp^n) \cap U(N)_v$. At the expense of possibly enlarging the set~$S_f$, note that for all~$v \notin S \cup \{v_\cp\}$, the subgroup~$U(N,\cp^n)_v$ is a hyperspecial maximal compact subgroup of~$U(N)_v$. This gives level structures on the quasisplit group~$U(N)$. If $H =U(N_1) \times ... \times U(N_r)$ is a product of quasisplit unitary groups, we define level subgroups~$H(\cp^n) = U(N_1, \cp^n)  \times ... \times U(N_r, \cp^n)$.  

Let~$(G,\xi)$ be an
inner form of~$U(N,F)$ defined with respect to a Hermitian inner product and with prescribed signatures~$U(a_v,b_v)$ at the archimedean places. We require that~$G_v$ be compact at the archimedean places contained in~$S_0$: this ensures that the group~$G$ is anisotropic. Following Proposition~\ref{classification of global inner forms}, if~$N$ is odd, the group~$G$ can be chosen so that~$G_v$ is quasisplit at all finite places. If $N$ is even, then $G$ is determined by choosing at most one place $v \in S_f$, up to again enlarging $S_f$. Once that choice is made, the group $G$ can be chosen to be quasisplit away from $\{v\} \cup S_\infty$. In both cases, this group $G$ is realized as an inner form $(G, \xi)$ as in Section \ref{section on inner forms}.

For each 
finite $v \notin S_f$, the inner twist induces isomorphisms $\xi_v: G_v \simeq U(N)_v$.
For each natural number~$n$, we fix a compact subgroup~$K(\cp^n) = \prod_v K_v(\cp^n)$ of~$G(\BAF)$ 
as follows:  at all finite~$v \notin S$, we
let~$K_v(\cp^n) = \xi_v^{-1}(U(N,\cp^n)_v)$; at $v \in S_f$, 
the subgroup~$K_v(\cp^n)$ is an arbitrary open compact 
subgroup fixed once and for all independently of $n$; at the archimedean places we let $K_v(\cp^n) \simeq U_{a_v}(\BR) \times U_{b_v}(\BR)$
be a maximal compact subgroup. Let~$K_f(\cp^n) = \prod_{v<\infty} K_v(\cp^n)$ and $K_\infty(\cp^n) = \prod_{v \mid \infty} K_v(\cp^n)$.   We may use the notation~$K_v$ instead of~$K_v(\cp^n)$ for $v \neq v_\cp$.  We extend these definitions to products of unitary groups.

We now define the (cocompact since $G$ is anisotropic) lattices $$ \Gamma(\cp^n) := G(F) \cap K_f(\cp^n).$$ Recall that $G_\infty = \prod_{v \mid \infty} G_v$ and let $X_G = G_\infty /K_\infty Z_{G_\infty}$. Assume that $G_\infty$ has at least one noncompact factor. The diagonal embedding~$ \Gamma(\cp^n) \hookrightarrow \prod_{v \mid \infty} G_v$ induces an action~$\Gamma(\cp^n) \curvearrowright X_G$, and we let~$X(\cp^n) := \Gamma(\cp^n) \dom X_G$. We start by comparing them to their disconnected counterparts realized as adelic double quotients. Let \[Y(\cp^n) = G(F) \dom G(\BAF) /K(\cp^n) Z_G(\BAF). \] The quotient $Y(\cp^n)$ is a disjoint union of finitely many connected locally symmetric spaces, each associated to a conjugate ot $K(\cp^n)$. In particular, the summand corresponding to $K(\cp^n)$ is $X(\cp^n)$. 

\begin{proposition} \label{proposition on the set of components}
	Let $G$ be an inner form of $U(N)$ and $Y(\cp^n)$ be defined as above. The cardinality of the set of components $\pi_0(Y(\cp^n))$ is bounded independently of~$n$. 
\end{proposition}
\begin{proof}
	We adapt an argument from \cite[\S 2]{MR15}. Considering~$G$ as a subgroup of~$GL_N/E$, let $\det: G \to U(1,E/F)$ be the determinant map and let $G^1 = \ker (\det)$. This map induces a fibering of $Y(\cp^n)$ over \[ U(1,F) \dom U(1,\BAF) / \det(Z(\BAF)K(\cp^n)).  \] The fibers are adelic double quotients for the group $G^1$, which is simply connected and has at least one noncompact factor at infinity. So by \cite[7.12]{PR94}, the group~$G^1$ satisfies strong approximation with respect to the set~$S_\infty$ and~$G^1(F) $ is dense in~$G^1(\BAFf)$, making the fibers connected. Thus we find that \[ \pi_0(Y(\cp^n)) \simeq U(1,F) \dom U(1,\BAF) / \det(Z(\BAF)K(\cp^n)) = E^1 \dom \BAE^1 / \det(Z(\BAF)K(\cp^n)). \] Now the image $\det (Z(\BAF)) $ is the subgroup $(\BAE^1)^N$ of $\BAE^1$. For each finite place $w$, the factor corresponding to $E_w$ in the quotient $\BAE^1/(\BAE^1)^N$ is a finite set. It follows that by increasing the level in powers of a single prime $\cp$, one can only produce a bounded number of components. 
\end{proof}

We now fix a unitary irreducible admissible representation $\pi_{\infty} = \ten_{v\mid \infty} \pi_v$ of $G_{\infty}$ with trivial central character and such that $\pi_v$ is the trivial representation if $G_v$ is compact. Denote \begin{equation}
	\label{definitionmPiInfty} m(\pi_{\infty}, \cp^n) := \dim \Hom_{G_\infty} (\pi_\infty,L^2(\Gamma(\cp^n) \dom G_\infty)).
\end{equation} Since $X(\cp^n)$ is one of the connected components of $Y(\cp^n)$, we have \begin{align}
\label{adelic multiplicity} 
m(\pi_\infty, \cp^n) &\leq  \dim\Hom_{G_\infty}(\pi_\infty,L^2(Y(\cp^n))) \\
&=  \sum_{\pi = \pi_\infty \ten \pi_f} m(\pi) \dim \pi_f^{K_f(\cp^n)}. \nonumber
\end{align}

 We will be interested in the asymptotics of the multiplicities $m(\pi_\infty, \cp^n)$ as $n \to \infty$.  
\subsection{Choice of Test Functions.} \label{Choice of Test Functions}
We define test functions whose traces will compute the multiplicity of archimedean representations at level~$\cp^n$. Recall that~$\mu_v$ denotes the Haar measure on~$G_v$. 
\begin{definition}
	At each finite place~$v$, let~$f_v(\cp^n):= 1_{K_v(\cp^n)}/\mu_v(K_v(\cp^n)).$
\end{definition}
\begin{definition}
	Let~$v \in S_0$ be an archimedean place such that~$G_v$ is compact. Let~$f_v(\cp^n)$ be equal to the constant function~$f_v = \mu_v(G_v)^{-1}.$
\end{definition}
The traces of these test functions count the dimension of spaces of~$K(\cp^n)$-fixed vectors. At~$v \in S_0$, they only detect the trivial representation and have vanishing trace on all other representations of~$G_v$. We want functions that play the same role at the non-compact archimedean places: they should detect representations~$\pi_v$ contained in a specific subset~$\Pi^0_{\psi_v} \subset \Pi_{\psi_v}$ and vanish on $\Pi_{\psi_v} \setminus \Pi^0_{\psi_v}$. The key is that we will only be working with Arthur packets attached to parameters~$\psi$ all having one specific~$\psi_\infty$. As such, the test function at an infinite place~$v$ only needs to isolate~$\pi_{v} \in \Pi^0_v$ from the other finitely many representations in the same packet. 
\begin{lemma}
	Let~$\psi_v$ be a local Arthur parameter with associated Arthur packet~$\Pi_{\psi_v}$. Fix a subset  $\Pi^0_{\psi_v} \subset \Pi_{\psi_v}$. Then there exists a function~$f_v^0 \in \mathscr{H}(G_v)$ such that \[ \tr \pi_v(f^0_v) = \begin{cases}
	1 &  \pi_v = \Pi^0_{\psi_v} \\ 
	0 & \text{otherwise,}
	\end{cases}\quad \quad  \pi_v \in \Pi_v. \]
\end{lemma}
\begin{proof}
	This follows directly from linear independence of characters for admissible representations. If~$v$ is archimedean this was proved by Harish-Chandra in~\cite{HC}. \end{proof}

\begin{definition}
	Let~$v$ be a non-compact archimedean place, let $\psi_v$ be an Arthur parameter and fix a subset $\Pi^0_{\psi_v} \subset \Pi_{\psi_{v}}$. Let $f_v(\cp^n) = f_{v}(\cp^n, \Pi_{\psi_v}^0)$ be the function $f^0_{v}$ described above.
\end{definition}

\begin{definition}\label{definition of test function}
	Let the function $f(\cp^n)$ be defined as $f(\cp^n) = \prod_vf_v(\cp^n_v)$. We will also denote $f_f(\cp^n) = \prod_{v \nmid \infty } f_v(\cp^n_v)$. 
\end{definition}

Given a choice of $\psi_\infty$ and $\Pi^0_{\psi_\infty}$, the function $f(\cp^n)$ satisfies the assumption of Theorem \ref{main theorem}: it is  factorizable and has nonnegative trace on $\pi \in \Pi_{\psi}$. 

\begin{prop} \label{Rdisc computes K-fixed vectors}
	Let $\psi \in \PsiG$. For each $v \in S_\infty \setminus S_0$, fix a subset $\Pi^0_{\psi_v}$ and a corresponding function $f(\cp^n) = f(\cp^n, \Pi^0_{\psi_v})$. Then we have  \[ \tr \Rdisci(f(\cp^n)) = \sum_{\pi} m(\pi) \dim \pi_f^{K_f(\cp^n)}  \] where the sum is taken over representations $\pi = (\ten_{v \mid \infty} \pi_v) \ten \pi_f \in \Pi_{\psi}$ such that for archimedean $v$, the representation $\pi_v$ is trivial if $v \in S_0$ and $\pi_v \in \Pi^0_v$ otherwise. 
\end{prop}
\begin{proof}
	As stated in Theorem \ref{Rdisc computes the traces on the packet.}, the  distribution $\tr \Rdisci(f)$ computes the sum of~$\tr \pi(f) = \prod_v \tr \pi_v(f_v)$ over all representations in the packet $\Pi_{\psi}$. At the finite places, the trace of convolution by the characteristic function of a compact open subgroup $K_v$ is equal to the product~$\mu_v(K_v) \cdot \dim \pi_v^{K_v}$. For archimedean places~$v \in S_0$, the representations~$\pi_v$ are finite-dimensional so the only representation with a~$K_v$-fixed vector is the trivial  representation. At~$v \in S_\infty \setminus S_0$,  the function~$f_{v}(\cp^n)$ was chosen precisely to detect~$\pi_v \in  \Pi^0_{\psi_v}$. 
\end{proof}

The key input allowing us to compare multiplicity growth on $G$ and $H \in \mE(H)$ is a fundamental lemma for congruence subgroups, proved by Ferrari~\cite{Fe}.

\begin{theorem}[Theorem 3.2.3, \cite{Fe}]\label{Ferrari}
	Let~$\cp$ be a prime of~$F$ with localization $F_{v_\cp}$ and residue field $k_\cp$. Let $\norm(\cp)$ be the cardinality of $k_\cp$ and let $p$ be its characteristic. Assume that~$p > N^2[F:\BQ]+1$. Let~$H \in \mE(G)$, and ~$d(G,H) = \frac{\dim G - \dim H}{2}$. Then the functions $$ f_{v_{\cp}}(\cp^n) = \frac{1_{K_{v_\cp}(\cp^n)}}{{\mu_{v_\cp}}(K_{v_\cp}(\cp^n))} \quad \text{and} \quad   f^H_{v_\cp}(\cp^n) = \norm(\cp)^{-n\cdot d(G,H)}\frac{1_{K_{v_\cp}(\cp^n)^H}}{{\mu_{v_\cp}} (K_{v_\cp}(\cp^n)^H)} $$ form a transfer pair. 
\end{theorem}

\subsection{Adaptation of Previous Limit Multiplicity Results.} \label{section on others' limit multiplicity results} Here, we collect all the results so far and import known upper bounds from the literature to prove our main limit multiplicity results. We start with a discussion of central characters. 

\subsubsection{Central Character Data.}  Our initial discussion of the discrete spectrum in~\ref{subsection on automorphic representations} included the prescription of a subgroup~$\fX \subset Z_G(\BAF)$, and we recalled in Proposition~\ref{proposition on central character} that in the case where~$G = U_{E/F}(N)$, the central character of representations in a packet $\Pi_{\psi}$ is determined explicitly in terms of $\psi$. We now need to extend this discussion to central characters of~$H \in \mE(G)$. For this, we will denote~$\fX_G = Z_G(\BAF)$, and~$\fX_H = Z_H(\BAF)$. As described in \cite[\S 3.2]{A}, the group~$\fX_G$ can be viewed canonically as a subgroup of~$\fX_H$, and we can speak of~$(\fX_G, \omega)$ as a central character datum of $H$, though it is not properly speaking a character~$Z_H(\BAF)$. This can be extended inductively to $H \in \mE\mH(G)$. From Proposition~\ref{proposition on central character} and the definition of the embeddings in~\eqref{embedding of endoscopic groups}, we get the following.
\begin{lemma} \label{lemma on restriction of central characters}
Let $G$ be a unitary group and $(H, \xi_H) \in \mE(G)$. Let $\psi \in \PsiG$ be associated to the central character datum $(\fX_G, \omega)$, and be such that~$\psi = \xi_H \circ \psi^H$.  Then the central character datum~$(\fX_G, \omega')$ associated to~$\psi^H$ is determined by~$\omega$ and~$\xi_H$.
\end{lemma}

\subsubsection{Sets of Parameters}
\begin{definition}
	Let $G$ be a reductive group, $\psi_{\infty}$ a parameter of $G_\infty$, and $(\mathfrak{X}, \omega)$ a central character. We denote by $\Psi(G,\psi_{\infty}, \omega)$ the set of $\psi \in \Psi(G)$ such that $(\psi)_\infty = \psi_\infty$ and such that the associated representations in the packet of $\psi$ have central character $\omega$.
\end{definition}
\begin{definition}
Letting $G$ be as above, and $(H, \xi) \in \mE(G)$ or $H = G$, we define for $ f \in \mathscr{H}(H)$,  \begin{equation} \label{EqDefIDiscPiInf}
	I^H_{\text{disc},\psi_\infty, \omega}(f) = \sum_{\psi \in \Psi(G, \psi_\infty, \omega)} I^H_{\text{disc}, \psi}(f)
\end{equation}  
\end{definition}

We will need to rewrite the right-hand side of \eqref{EqDefIDiscPiInf} as a sum over parameters of $H$. By definition, for each $\psi \in \Psi(G,\eta)$, we have \[ I^H_{\text{disc}, \psi}(f) = \sum_{c^H \stackrel{\eta \circ \xi}{\to}c(\psi)} I^H_{\text{disc},c^H}(f), \] for $c^H \in \mathcal{C}(H)$. But following the main theorem of the spectral expansion of the trace formula \cite[Prop. 4.3.4]{Mok} applied to each of the simple factors of $H$, shows that $I^H_{\text{disc},c^H}(f) = 0$ unless $c^H = c(\psi^{H,N})$ is attached to a parameter $\psi^H = (\tilde \psi^H, \psi^{N,H})$. By the assumption $c^H \stackrel{\eta \circ \xi}{\to}c(\psi)$, we mush have $\psi^{H,N} = \psi^N$, and $\tilde\psi$ must factor through $^LH$. Thus we can rewrite

\[ \sum_{\psi \in \Psi(G, \psi_\infty, \omega)} I^H_{\text{disc}, \psi}(f) =  \sum_{\psi  \in \Psi(G, \psi_\infty, \omega)} \sum_{\psi^H \mapsto \psi }I^H_{\text{disc},\psi^H,\omega}(f).\]

In the arguments of this section, we will work with three families of groups, and three sets of parameters, which we describe now. 
\begin{itemize}
	\item[-] The group~$G \in \tilde{\mE}_{sim}(N)$ is the inner form of~$U_{E/F}(N)$ for which we ultimately want to produce bounds. These will be obtained in Theorem~\ref{main theorem 2} by taking a sum over~$\Psi(G, \psi_\infty, 1)$, for the trivial central character of~$Z_G$. 
	\item[-] The group~$H = H_\psi = U(N_1) \times U(N_2)$ with~$N_1 + N_2 = N$ belongs to the endoscopic datum~$(H, \xi, s) \in \mE(G)$ whose stable trace gives the upper bounds in Theorem~\ref{main theorem}. In the lead-up to Theorem \ref{main theorem}, for each~$\psi \in \PsiG$, we singled out a parameter~$\psi^{H}$ such that~$\xi \circ \psi^H = \psi$. Through this choice, the parameter~$\psi_{\infty}$ of~$G_\infty$ determines a unique parameter~$\psi^H_\infty$ of~$H_\infty$. In Proposition~\ref{prop on hyperendoscopic upper bounds}, the sum will be taken over the set~$\Psi(H, \psi^H_\infty, \omega)$ for a suitable central character~$\omega$.
	\item[-] The difference between the distribution giving the upper bounds in Theorem \ref{main theorem} and $I^H_{\text{disc},\psi^H}$ is expressed in Corollary \ref{corollary on hyperendoscopy S(psi,sH,f)} in terms of hyperendoscopic data $(H_q,\xi_q)$ for $H$. We will consider parameters such that $\xi_{q}(\psi^{H_q})  \in \Pi(H, \psi^H_\infty, \omega)$. We have $\psi^{H_q}_\infty = \psi^{H_q^1}_\infty \times \psi^{H_q^2}_\infty$ and we will give upper bounds on the multiplicities of representations associated to each of these factors in Propositions \ref{Proposition Savin with Central Character} and \ref{proposition growth of characters}. 
\end{itemize}

\subsubsection{Upper Bounds for Hyperendoscopic Groups.}
We start by adapting limit multiplicity results of Savin~\cite{Sa89}, which will form the basis for our inductive proof. Since Savin's result applies to semisimple groups, we pay attention to the central characters and components of locally symmetric spaces. We first give bounds for bounded parameters. The result is stated in terms of any $G$ and $H$, but will be specialized to $G = U(N_2)$ and $H = H^2_q$. 

\begin{proposition} \label{Proposition Savin with Central Character}

	Let $G \in \tilde{\mE}_{sim}(N)$ and $\mH \in \mH\mE(G)$ be a hyperendoscopic group. Let~$\psi_{\infty} \in \Psi(G_\infty)$ be a bounded parameter with regular infinitesimal character. Let~$(\fX_G, \omega)$ be a central character datum for $G$ such that~$\omega \mid_{(\fX_G \cap G_\infty)}$ is the central character associated to~$\psi_\infty$ by Proposition \ref{proposition on central character}. Let $v_\cp$ be an unramified finite place of $F$, associated to the prime $\cp$, and let $f(\cp^n) = \prod_v f_v(\cp^n) \in \mathscr{H}(H)$ be such that \begin{itemize}
		\item $f_v(\cp^n)$ is independent of $n$ if $v \neq v_\cp$
		\item $f_{v_\cp}(\cp^n) =  \frac{1_{K(\cp^n)}}{\mu(K(\cp^n))}$, for $K(\cp^n)$ as in Section \ref{section on level structures}. 
		\item $f(\cp^n)$ satisfy the assumptions of Theorem \ref{main theorem}
	\end{itemize} Then \[|I^H_{\text{disc}, \psi_{\infty}, \omega}(f(\cp^n))| \ll Nm(\cp^n)^{\dim H-1}.\] 
\end{proposition}
\begin{proof} Since the infinitesimal character of $\psi_\infty$ is regular, we can equate 
	
	\[ I^H_{\text{disc},\psi_\infty, \omega} (f(\cp^n)) = \tr R^H_{\text{disc}, \psi_\infty, \omega }(f(\cp^n)) = \sum_{\psi^H_\infty \mapsto \psi_\infty} \sum_{\psi^H \in \Psi(H, \psi^H_\infty, \omega')} \sum_{\pi \in \Pi_{\psi^H}} m(\pi) \tr \pi (f(\cp^n)), \] where $(\mathfrak{X}, \omega')$ is the central character datum associated to $\psi$ as in Lemma \ref{lemma on restriction of central characters}. Since the first sum is finite, we will bound  \[  \left|\sum_{\psi^H \in \Psi(H, \psi^H_\infty, \omega')} \sum_{\pi \in \Pi_{\psi^H}} m(\pi) \tr \pi (f(\cp^n))\right|. \]

	For each $\pi$, we have $ \tr \pi (f(\cp^n)) = \prod_v \tr \pi_v(f_v(\cp^n))$.  At $v \mid \infty$, the packet is always the same, so~$|\tr \pi_\infty(f_\infty)|$ is uniformly bounded. For each finite~$v$, there is an open compact subgroup~$K'_v \subset G_v$, depending on~$n$ only if~$v = v_\cp$, such that~$|\tr\pi_v f_v(\cp^n)| \neq 0 \implies \dim \pi_v^{K'_v} \neq 0$. Indeed, since $f_v(\cp^n)$ is~$K_v$-finite, where~$K_v$ is a maximal compact subgroup, there is a subgroup~$K'_v \leq K_v$ of finite index such that~$f_v(\cp^n)$ is~$K'_v$-invariant, so that convolution by $f_v$ is a projection onto $\pi_v^{K'_v}$. At all but finitely many places the group~$K_v$ is hyperspecial, we have~$f_v(\cp^n) = 1_{K_v}$ and we can take~$K'_v = K_v$. Thus we have $|\tr \pi_vf_v|<C(f_v)\dim \pi^{K_v'}$ and by Bernstein's uniform admissibility~\cite{Ber74all}, the right-hand side is bounded uniformly, with~$C(f_v(\cp^n)) = 1$ at~$v \notin S$.
At $v = v_p$, we have $K_v' = K_v(\cp^n)$. Let $K'(\cp^n) = \prod_{v < \infty} K'_v$. From our restriction on the central character, we thus have \[ \left|\sum_{\Psi(H, \psi^H_\infty, \omega')} \sum_{\pi \in \Pi_{\psi^H}} m(\pi) \tr \pi (f(\cp^n))\right| \leq C(\psi_\infty, S) \sum_{\substack{\pi : \pi_\infty \in \Pi_{\psi^H_\infty}\\ \omega(\pi) = \omega'}} m(\pi) \dim \pi_f^{K'(\cp^n)}, \] where $C(\psi_\infty, S)$ is a constant depending only on $\psi_\infty$ and $S$.  

Since $\psi_\infty^H$ is bounded, any representation $\pi_\infty \in \Pi_{\psi^H_\infty}$ is tempered, which implies that $\pi \in \Pi_\psi^H$ occur in the cuspidal part of the discrete spectrum~\cite[Theorem 4.3]{Wa84}.
Thus for each $\pi_\infty \in \Pi_{\psi^H_{\infty}}$ we have  \begin{equation*} \sum_{\substack{\pi = \pi_\infty\cdot \pi_f\\ \omega(\pi) = \omega'}} m(\pi) \dim \pi_f^{K'(\cp^n)} \leq \dim \Hom_{H_\infty}(\pi_{\infty}, L^2_{\text{cusp}}(H(F)\dom H(\BAF),\omega')^{K'(\cp^n)}) \end{equation*}

 The right-hand side of the inequality is equal to \begin{equation}\label{eqn central character} m(\pi_{\infty}, \cp^n, \omega') := \dim\Hom_{H_\infty}(\pi_\infty, L^2_{\text{cusp}}(H(F)\dom H(\BAF)/K'(\cp^n), \omega')). \end{equation}
The space~$Y^*_H(\cp^n) := H(F) \dom H(\BAF)/K'(\cp^n)$ carries commuting actions of~$H_\infty$ and~$Z_H(\BAF)$, inducing representations on $L^2_{\text{cusp}}(Y^*_H(\cp^n)).$ For $n$ large enough, the character $\omega'$ is trivial on~$\fX_G\cap K'(\cp^n)$, and thus appears in the representation of~$\fX_G$ on~$L^2_\text{cusp}(Y^*_H(\cp^n))$. It is this $\omega'$-isotypic subspace that we denote by~$L^2_{\text{cusp}}(Y^*_H(\cp^n), \omega')$. 

To bound~$m(\pi_\infty, \cp^n, \omega')$, consider first the case where the central character datum for ~$\fX_H$ is trivial: this setup is similar to that of Proposition \ref{proposition on the set of components}.  We have~$H = U(N_1) \times ... \times U(N_r)$; let $H^1 = SU(N_1) \times ... \times SU(N_r)$. The representation~$\pi_{\infty}$ of~$H_\infty$ restricts to an irreducible representation~$\rho_\infty$ of~$H^1_\infty$, see \cite[\S 2]{AP06certain}. Let \[X_H(\cp^n) = H^1(F) \dom H^1(\BA)/ K^1(\cp^n), \quad K^1(\cp^n) = K'(\cp^n) \cap G^1(\BA).\]  
The group~$H^1$ is simply connected and has no compact factors at infinity, so~$X_H(\cp^n)$ is connected~\cite{PR94}.  Following a result of Savin~\cite{Sa89}, we have \[m(\rho_\infty, \cp^n):=\dim \Hom_{H^1_\infty}(\rho_\infty, L^2_{\text{cusp}}(X_H(\cp^n))) \asymp \Vol(X_H(\cp^n)) \asymp Nm(\cp)^{n \cdot \dim H^1}.\] We now consider general central characters. The space~$Y^*_H(\cp^n)$ is a disjoint union of finitely many locally symmetric spaces, associated to conjugates of~$K^1(\cp^n)$, and the theorem of Savin applies to each of them. Let~$T = H/H^1$, and let~$\nu$ denote the quotient map, through which all central characters factor, c.f. Proposition~\ref{proposition on central character}. Following \cite[2.7.1]{De71}, the set~$\pi_0(Y^*_H(\cp^n))$ is a torsor for the finite group \[  T_{\cp^n}:=T(\BAF)/T(F)\nu(K'(\cp^n)). \] Denote by $\fX_{G,\cp^n}$ the image of~$\fX_G$ in this quotient. The action of~$\fX_{G,\cp^n}$ on~$\pi_0(Y^H(\cp^n))$ is induced by multiplication in $T_{\cp^n}$, thus~$\pi_0(Y^*_H(\cp^n))$ is a finite union of $[T_{\cp^n}:\fX_{G,\cp^n}]$ many principal homogeneous spaces for~$\fX_{G,\cp^n}$. Thus as a~$\fX_{G,\cp^n}$-representation, the space $\Hom(\pi_\infty, L^2_{\text{cusp}}(Y^*_H(\cp^n)))$ caries finitely many copies of the regular representation of~$\fX_{G,\cp^n}$, and all characters of~$\fX_G$ factoring through $\fX_{G,\cp^n}$ occur with equal multiplicity. The group~$\fX_G$ is the adelic points of torus of diagonal matrices isomorphic to $U(1)$, and~$T \simeq U(1)^r$. So each character~$\omega'$ of~$\fX_G$ factoring through~$\fX_{G,\cp^n}$ does so with multiplicity \[ m(\pi_\infty, \cp^n, \omega') = m(\rho_\infty, \cp^n) [T_{\cp^n} : T^H_{\cp^n}] = m(\rho_\infty, \cp^n)\frac{[T(1):T(\cp^n)]}{[\fX_{G,1}:\fX_{G,\cp^n}]} \asymp  \norm(\cp^n)^{\dim H -1}. \] Summing over all~$\pi_\infty$ in $\Pi_{\psi_{\infty}}$, we conclude.  \end{proof}

We now give bounds for parameters where $\psi(SL_2)$ is maximally large. In the final proof, the group $G$ will be specialized to $U(N_1)$.  

\begin{proposition} \label{proposition growth of characters}
	Let~$G \in \tilde{\mE}_{sim}(N)$, and~$\psi_\infty \in \Psi(G_\infty)$. Let $f(\cp^n) = \prod_{v} f_v(\cp^n)$ satisfy the same assumptions as in Proposition \ref{Proposition Savin with Central Character}.  Let~$\psi_{\infty} \in \Psi(G_\infty)$ be a parameter with regular infinitesimal character and such that~$\psi_{\infty}\mid_{SL_2} = \nu(N)$. Let~$(\fX_G, \omega)$ be a central character datum. Then there is a constant~$M$ depending only on~$G, \psi_\infty$, and the set~$S$ of bad places (and in particular neither on $n$ nor on $\omega$) such that   \[ |I^G_{\text{disc}, \psi_{\infty}, \omega}(f(\cp^n))| < M.
\]
\end{proposition}
\begin{proof}
The proof is a simplified version of that of Proposition \ref{Proposition Savin with Central Character}. The restriction of the infinitesimal character and on the possible representations at infinity gives \[ |I^G_{\text{disc}, \psi_{\infty}, \omega}(f(\cp^n))| \leq C(\psi_{\infty}, S) \sum_{\substack{\pi : \pi_\infty \in \Pi_{\psi_\infty}\\ \omega(\pi) = \omega}} m(\pi) \dim \pi_f^{K'(\cp^n)}. \] The assumption on the Arthur $SL_2$ implies that the representations $\pi_{\infty} \in \Pi_{\psi_\infty}$ are one-dimensional, see e.g. \cite[\S 5]{AUC}. Thus they factor through the determinant map~$\nu$, and, as above, through the action of the quotient~$T(\BAF)/T(F)$. It follows that in this case the multiplicity~$m(\pi_\infty, \cp^n, \chi)$ is bounded above by~$|T(\BAF)/T(F)\nu(K(\cp^n))|$. Recall here that $\fX_G = Z(\BAF)$. If~$\omega$ were trivial, then the representations contributing to~$m(\pi_{\infty}, \cp^n, \omega)$ would be bounded above by the size of the quotient $|T(\BAF)/T(\BQ)\nu(K(\cp^n)\cdot Z(\BAF))|$, which we showed in Proposition \ref{proposition on central character} to be bounded independently of $n$. But by the proof of Proposition \ref{Proposition Savin with Central Character}, the representation of $\fX_G$ on $\Hom(\pi_\infty, L^2_{\text{cusp}}(Y^*H(\cp^n)))$ factors through a sum of copies of the regular representation of a finite quotient $\fX_{G,\cp^n}$. As such, all characters of $\fX_G$ appearing in the quotient do so with equal multiplicity. Thus the bound $M$ also holds for $\omega$. 
\end{proof}

\subsection{Limit Multiplicity for $G$.} \label{section on our limit multiplicity results}
Before we assemble the results for various endoscopic groups, we bound the number of central character data $(\fX_H, \omega)$ of a given level and restriction to $\fX_G$. 
\begin{lemma} \label{lemma on counting central characters}
	Let $H \in \mE(G)$. Let $(\mathfrak{X}_G, \omega)$ be a central character datum for $G$. For each $n$, fix a level structure $K^H_f(\cp^n)$ as in section~\ref{section on level structures}. Define \[ \Xi(\omega, \cp^n) = \{ (\fX_H,\omega_H) \;:\; \omega_H\mid_{\fX_G} = \omega,\; \omega_H(\fX_H \cap K^H_f(\cp^n)) = \omega_H(\fX_H \cap Z_H(F)) = 1 \}. \] Then we have $ |\Xi(\omega, \cp^n)| \ll \norm(\cp^n). $
\end{lemma}
\begin{proof}
Central characters of~$H = U_{E/F}(N_1) \times U_{E/F}(N_2)$ are products~$\omega_H = \omega_1 \times \omega_2$ of characters of the respective centers. The condition upon restriction to $Z_H(F)$ implies that these are of the form $\omega_i = \theta_i \circ \det$ for $\theta_i$ a Hecke character of $\BA^1_{E}$. Given a choice of~$\theta_1$, the condition that~$\omega_H \mid_{\mathfrak{X}_G} = \omega$ restricts~$\theta_2$ to at most $N_2$ different characters. The restriction on conductor thus implies that $|\Xi(\chi, \cp^n)| \ll |\Xi(\fa\cp^n)| $, where $\Xi(\fa\cp^n)$ consists of Hecke characters $\theta_1$ of $\BAE^1/E^1$ whose conductor divides $\fa\cp^n$; the presence of a conductor away from $\cp$ comes from the possibility that at places $v \in S_f$, the (fixed) subgroup $K^H_v(\cp^n)$ is not maximal. The number of such characters grows like $\norm(\cp^n)$.
\end{proof}

We now assemble the results of Section \ref{section on others' limit multiplicity results} to give upper bounds for the contribution of parameters where all but one summand have trivial Arthur $SL_2$. We start by bounding the contribution of each hyperendoscopic group of $H_\psi$. 

\begin{proposition}\label{prop on hyperendoscopic upper bounds} Let~$G \in \tilde{\mE}_{sim}(N)$, and~$H = (U_{E/F}(N_1) \times U_{E/F}(N_2), \xi) \in \mE(G)$. Let $\psi_\infty$ and $\psi^H_\infty$ be such that 
\begin{itemize}
\item[(i)] $\psi_\infty = \xi \circ \psi^H_\infty$, 
\item[(ii)]$\psi_\infty = \psi_{\infty,1} \oplus \psi_{\infty,2}$ with $ \psi_{\infty,1}\mid_{SL_2} = \nu(N_1)$ and $\psi_{\infty,2}\mid_{SL_2} = \nu(1)^{N_2}$,
\item[(iii)] each $\psi_{i,\infty}$ factors through ${^L}U(N_i)$. 
\end{itemize} 
Let $(\fX_G, \omega)$ be a central character datum for $H$ consistent with to $\psi^H_\infty$. Assume that~$\cp$ is large enough to apply the results of Theorem~\ref{Ferrari}. Let~$H_q \in \mH\mE(H)$, and let $f(\cp^n)$ be the sequence of test functions defined in \ref{definition of test function}. Then \[ \left|I^{H_q}_{\text{disc}, \psi^H_\infty, \omega}(f^{H_q}(\cp^n)) \right|  \ll Nm(\cp^n)^{N(N-N_1)}. \]
\end{proposition}
\begin{proof}
Since $H_q$ is a hyperendoscopic group of~$H$ we have~$H_q = H_q^1 \times H_q^2$, where~$\psi_i$ factors through $H_q^i$. The growth rate in the theorem is defined up to constants, so we can assume that $f^{H_q}(\cp^n) = f^{H^1_q}(\cp^n) \times f^{H^2_q}(\cp^n)$; this is true locally almost everywhere by the Fundamental Lemma. Indeed, at all $v \notin S \cup \{v_\cp\}$, the function~$f_v^{H_q}(\cp^n)$ can be taken to be the characteristic function of a hyperspecial maximal compact subgroup of $H_{q,v}$.  At $v = v_\cp$, we iterate the conclusion of  Theorem \ref{Ferrari} to get \begin{equation}\label{eq f vs phi }f^{H_q}_{v_\cp}(\cp^n) = \norm(\cp)^{-n\cdot d(G,H_q)} \frac{1_{K_{v_\cp}(\cp^n)^{H_q}}}{\mu(K_{v_\cp}(\cp^n))}:=\frac{\norm(\cp)^{-n\cdot d(G,H_q)}}{\mu(K_{v_p})(\cp^n)/\mu(K_{v_\cp}(\cp^n)^{H_q})} \phi_{v_\cp}(\cp^n).\end{equation} Write $\phi_{v_\cp}(\cp^n) = \phi^1_{v_\cp}(\cp^n) \times \phi^2_{v_\cp}(\cp^n)$, and for $i=1,2$, let  \[\phi^{H^i_q}(\cp^n) = \phi^i_{v_\cp}(\cp^n) \cdot \prod_{v \neq v_\cp} f^{H^i_q}_v(\cp^n), \quad \phi^{H_q}(\cp^n) = \phi^{H^1_q}(\cp^n) \times \phi^{H_q^2}(\cp^n).\]

Each of the two functions $\phi^{H_q^i}(\cp^n)$ satisfies the identical assumptions of Propositions \ref{Proposition Savin with Central Character} and \ref{proposition growth of characters}. We also recall that $H = U(N_1) \times U(N_2)$, and we shorten $U(N_i) = H^i$. We also have $\psi^H = \psi_1 \oplus \psi_2$ with $\psi_i$ landing in ${^L}H^i$. Thus if we fix data $(\fX_{H^1}, \omega_1)$ and $(\fX_{H^2}, \omega_2)$ coming from $H^1$ and $H^2$ respectively, we find that \begin{align*}
\left|\sum_{\Psi(H,\psi^H_\infty, \omega_1 \times \omega_2)} I^{H_q}_{\text{disc}, \psi^H}(\phi^{H_q}(\cp^n)) 
\right| & \leq \sum_{ \Psi(H,\psi_\infty^H, \omega_1 \times \omega_2)} |I^{H_q^1}_{\text{disc}, \psi_{1}}(\phi^1(\cp^n))| \cdot |I^{H_q^2}_{\text{disc}, \psi_{2}}(\phi^2(\cp^n))| \\
&= |I^{H_q^1}_{\text{disc}, \psi_{\infty,1}, \omega_1}(\phi^1(\cp^n))| \times |I^{H_q^2}_{\text{disc}, \psi_{\infty,2}, \omega_2}(\phi^2(\cp^n))| \\ 
&\ll M \cdot Nm(\cp^n)^{\dim H_q^2 -1}. 
\end{align*}
The quantity in the left-hand side above isn't quite what we want to measure. First, we want to replace the choice of a pair of central characters $\omega_1 \times \omega_2$ by a sum over all parameters with central character datum $(\fX_G, \omega)$. In Lemma~\ref{lemma on counting central characters}, we saw that the number of products~$\omega_1 \times \omega_2$ of level~$\cp^n$ which restrict to~$\omega$ on~$\fX_G$ is~$\ll \norm(\cp^n)$.  Second, we slightly modify the test functions. From \eqref{eq f vs phi }, we have 

\[ f^{H_q}(\cp^n) = C(G, H_q, n)\phi^{H_q}(\cp^n) , \quad  C(G, H_q, n)\asymp \norm(\cp)^{n \cdot d(G,H_q)}. \] 

Thus combining our upper bounds with these modifications we obtain

\begin{align*} \left| I^{H_q}_{\text{disc}, \psi^H_\infty, \omega}(f^{H_q}(\cp^n)) \right|  &\ll \norm(\cp^n)^{(1+ d(G, H_q))}\left|\sum_{\psi^H \in \Psi(H,\psi_\infty^H, \omega_1 \times \omega_2)} I^{H_q}_{\text{disc}, \psi^H}(\phi^{H_q}(\cp^n)) 
\right| \\ 
& \ll \norm(\cp^n)^{d(G,H_q)+\dim H^2_q} \\& = \norm(\cp^n)^{\dim(G)/2 -\dim (H^1_q)/2 + \dim (H^2_q)/2}. \end{align*}
Recall that $\dim G = N^2$, and that since the dual group of $H^1_q$ receives an ~$N_1$-dimensional irreducible representation of~$SL_2$, we have~$\dim(H^1_q) = N_1^2$. Finally, it follows that $\dim(H^2_q) \leq (N-N_1)^2$, which gives us the desired upper bounds.
\end{proof}
\begin{remark}
	Note that the only situation in which this upper bound has a chance of being sharp is when $\dim H_q^2 = (N-N_1)^2$, i.e. when $H_q = H$. 
\end{remark}
We have now collected all the facts leading up to our limit multiplicity theorem. 
\begin{theorem} \label{main theorem 2}
	Let $\psi_{\infty}$ be an Arthur parameter with regular infinitesimal character, and such that~$\psi_{\infty}\mid_{SL_2(\BC)} = \nu(2k) \oplus \nu(1)^{N-2k}$. Let $(\mathfrak{X}_G,1)$ be the trivial central character. Fix~$\Pi^0_{\psi_\infty} \subset \Pi_{\psi_{\infty}}$. For each~$\psi \in \Psi(G, \psi_\infty, 1)$, let \[ \Pi^0_\psi = \{\pi = \ten'_v \pi_v \in \Pi_{\psi} \mid \pi_\infty \in \Pi^0_{\psi_\infty}. \} \] Then \begin{equation} \label{inequality in application 1} \sum_{\psi \in \Psi(G, \psi_\infty,1)} \sum_{\pi \in \Pi^0_\psi} m(\pi) \dim \pi_f^{K(\cp^n)} \ll \norm(\cp^n)^{N(N-2k)}. \end{equation}
\end{theorem}
\begin{proof}

	From Proposition~\ref{Rdisc computes K-fixed vectors}, we take $f(\cp^n)$ as in Section~\ref{Choice of Test Functions} and write \begin{align*} \sum_{\psi \in \Psi(G, \psi_\infty,1)}  \sum_{\pi \in \Pi^0_\psi} m(\pi) \dim \pi_f^{K(\cp^n)} &=  \sum_{\psi \in \Psi(G, \psi_\infty,1)}  \tr R_{\text{disc}, \psi}(f(\cp^n)) \\
	&= \sum_{\psi \in \Psi(G, \psi_\infty,1)}  I_{\text{disc}, \psi}(f(\cp^n))  & \text{(Theorem \ref{Bergeron-Clozel on parameters with reg inf char})} \\ 
	&\leq \sum_{\psi \in \Psi(G, \psi_\infty,1)}  C(\psi) S(\psi,s_{H_\psi}, f(\cp^n)), \end{align*}where the last inequality follows from the results of Section~\ref{Bounds by the Dominang Group}, where the notation $C(\psi)$ was defined, since~$f(\cp^n)$ takes only positive values. The group~$H_\psi$ is determined by the localization $\psi_v$ of $\psi$ at any place $v$, and in particular by $\psi_\infty$. Thus $H_\psi$ is the same for any~$\psi \in \Psi(G, \psi_\infty,1) $ since by definition they all localize to the same $\psi_\infty$. By the assumption on~$\psi_\infty$, we have~$H_\psi = U_{E/F}(2k) \times U_{E/F}(N-2k)$, and the parameters satisfy the assumptions of Proposition~\ref{prop on hyperendoscopic upper bounds}. To lighten the notation, we denote $H_\psi$ by $H$ and $s_\psi$ by $s_H$ for the end of the proof. 

The parameter $\psi^{H}$ corresponding to $s_{H}$ under the bijection of Lemma \ref{lemma on endoscopic bijection}, we have shown in Corollary \ref{corollary on hyperendoscopy S(psi,sH,f)} that \[ S(\psi,s_H, f(\cp^n)) = \iota(G,H) \sum_{\mH \in \mH\mE(H,\psi)} \iota(\mH) I^{H_q}_{\text{disc},\psi^H}(f^{H_q}(\cp^n)). \] For each summand on the right-hand side, we sum over $\Psi(H, \psi^H, \omega)$, where $(\mathfrak{X}_G, \omega)$ is determined by $\psi^H$ as in Lemma~\ref{lemma on restriction of central characters}. We then apply Proposition~\ref{prop on hyperendoscopic upper bounds} with~$N_1 = 2k$. Note that we have ensured in Lemma \ref{lemmainfcharremainsregular} that the infinitesimal character of the representations of~$H_{q, \infty}$ associated to all~$\psi_{\infty}^{H_q}$ are regular. This gives us the following bounds:  \[ \left|\sum_{\psi \in \Psi(H,\psi_\infty, 1)} I^{H_q}_{\text{disc}, \psi^H}(f^{H_q}(\cp^n)) \right|  \ll Nm(\cp^n)^{N(N-2k)}. \] We conclude by summing over the finitely many $H_q \in \mH\mE(H,\psi)$.  
\end{proof}

\section{Applications to Growth of Cohomology.} \label{section on cohomology}

We now apply the results of Section \ref{section on limit multiplicity} to cohomology of arithmetic groups. This section is concerned with local questions at infinity, and the notation is different from the rest of the paper. From now until Section \ref{last section on limit multiplicity}, $G$ will be a Lie group.

\subsection{Cohomological Representations} Given a Lie group $G$, let $\tilde{G}$ denote the unitary dual of $G$. 

\begin{theorem}[Matsushima's formula, \cite{Ma67}]
	Let~$G$ be a connected semisimple Lie group with maximal compact subgroup~$K$ and complexified Lie algebra~$\fg$. Let~$\Gamma \subset G$ be a cocompact lattice and let~$X_\Gamma = \Gamma \dom G /K$. For~$\pi \in \tilde{G}$, denote by~$m(\pi, \Gamma)$ the multiplicity of $\pi$ in the right-regular representation of~$G$ on~$L^2(\Gamma\dom G)$. Then:
	\[ \dim(H^i(X_\Gamma, \BC)) = \sum_{\pi \in \tilde{G}}m(\pi, \Gamma)\dim(H^i(\fg,K;\pi)).\] 
\end{theorem} 

The $H^i(\fg, K; \pi)$ which appear in the right-hand side are the so-called $(\fg,K)$ cohomology groups of~$\pi$.  We say that $\pi$ is \emph{cohomological} if~$H^*(\fg, K; \pi) \neq 0$; such representations were characterized by Vogan-Zuckerman \cite{VZ}. 

\begin{theorem}[\cite{VZ}] \label{Vogan-Zuckerman Theorem}
	Let~$G$, $\fg$ be as above. Let~$K$ a maximal compact subgroup of~$G$ and~$\fg = \fk \oplus \fs $ be the corresponding Cartan decomposition with $\fk$ the Lie algebra of $K$. The group~$G$ has finitely many cohomological representations $\pi$, and~$H^i(\fg, K; \pi) \neq 0$ if and only if: 
	\begin{itemize}
		\item[(i)] $\pi$ has the infinitesimal character of the trivial representation of~$G$, and; 
		\item[(ii)]  $\Hom_K(\pi, \wedge^i\fs) \neq 0$,
	\end{itemize} where the action of~$K$ on~$\wedge^i \fs$ is induced by the adjoint representation. 
	
\end{theorem}

The results apply only to semisimple groups: they are extended to~$U(a,b)$ in \cite{Vo97}, and condition (ii) above implies that cohomological representations have trivial central character. Below, we give a concrete parametrization of cohomological representations of~$U(a,b)$ in terms of refinements of partitions of $a+b$ which are compatible with the signature $(a,b)$. More details can be found in \cite{BC05} and \cite{GG20}. 

\subsubsection{Cohomological Representations and Ordered Bipartitions} \label{computation of cohomology}

In \cite{VZ}, cohomological representations $A_\fq$ are built from so-called $\theta$-stable parabolic subalgebras $\fq = \fl \oplus \fu$. In~\cite[\S 5]{BC05}, Bergeron--Clozel show that for $U(a,b)$, the data of the algebra~$\fq$ can be encoded in a choice of centralizing Levi subgroup~$L(\fq) = \prod U(a_i,b_i) \subset U(a,b)$ whose Lie algebra is~$\fl$. 
Thus~$\fq$'s are parameterized by ordered tuples \[ B = ((a_1,b_1),...,(a_r, b_r))\] of pairs of nonnegative integers with $\sum a_i = a$ and $\sum b_i = b$. We call these tuples $B$ \emph{ordered bipartitions} of $(a,b)$ and denote the associated Levi subgroup $L_B$, and the corresponding representation by $\pi_{B}$. 

The ordered bipartitions of $(a,b)$ \emph{almost} parametrize the cohomological representation of $U(a,b)$, but there is redundancy. Specifically, $\pi_B \simeq \pi_{B'}$ if $B'$ has adjacent pairs of the form $(a_1, 0),(a_2,0)$ (resp.~$(0,b_1)(0,b_2)$) which are collapsed into~$(a_1+a_2,0)$ (resp.~$(0,b_1+b_2)$) in $B$. We will say that an ordered bipartition is \emph{reduced} if all pairs in which one entry is zero are maximally broken up.  

\begin{example}The following ordered bipartition is not reduced: \[ ((3,1),(2,0),(1,0),(0,3)).\] It is associated to the same cohomological representation that  to the following reduced ordered bipartition: \[((3,1)(1,0)(1,0)(1,0)(0,1)(0,1)(0,1)).\]
\end{example}

The cohomology of $\pi_B$ can be expressed in terms of $B$.  

\begin{prop}[\cite{VZ}, Proposition 3.2] \label{prop degree computation}
	Let $\fq = \fl \oplus \fu$ be a $\theta$-stable parabolic subalgebra and let $\fg = \fk \oplus \fs$ be the Cartan decomposition. Let $R = \dim \fu  \cap \fs$. Then \[H^i(\fg,K, A_\fq) \simeq \Hom_{\fl \cap \fk}(\wedge^{i-R} \fs, \BC).\] 
\end{prop}

In particular, the smallest nonvanishing degree of cohomology of $A_\fq$ is $R$, which, writing $A_\fq = \pi_B$ and referring once more to \cite{BC05} and \cite{GG20}, is equal to  \begin{equation}
\label{degreeofcohomology} R = \frac{\dim(\fs)-\dim (\fs \cap \fl)}{2}  = ab- \sum_{i=1}^r a_ib_i. 
\end{equation}
In particular, if $a_ib_i = 0$ for all pairs, i.e. if $L$ is compact, then $\pi_B$ is a discrete series representation and only has cohomology in degree $ab$. 
\subsection{Arthur Parameters of Cohomological Representations} \label{Arthur Packets of Cohomological Representations}
We turn our attention to archimedean parameters~$\psi$ whose associated Arthur packets contain cohomological representations. These are obtained via embedding of~$L$-groups from parameters associated to the trivial representation of Levi subgroups of~$G = U(a,b)$. The packets associated to these parameters were constructed by Adams-Johnson~\cite{AJ} in conversation with work of Arthur~\cite{AUC}, in a language predating the current formulation of the endoscopic classification of representations.  Arancibia--Moeglin--Renard~\cite{AMR18} have shown that Adams-Johnson's construction yields the same packets as those appearing in the endoscopic classification in \cite{Mok} and ~\cite{KMSW}.  

To begin, note that there is a natural way to associate to an ordered bipartition~$B$ of~$(a,b)$ an ordered partition $P_B$ of $N$, namely by letting \[ P_B = (N_1,...,N_r), \quad N_i = a_i+b_i. \]
Let $B$ be an ordered bipartition, and $L_B$ be the associated Levi subgroup. Then $\hat{L} \simeq \prod_i GL_{N_i}(\BC) \hookrightarrow \hat{G}$, is determined by $P_B$. The description of ${^L}L$, i.e. of the Galois action on $\hat{L}$, is given in \ref{section on L-groups of unitary groups}. 
Cohomological Arthur parameters depend on an embedding $\xi_{\hat{L},\hat{G}}: {^L}L \hookrightarrow {^L}G$ extending the map $\hat{L} \hookrightarrow \hat{G}$. To define $\xi_{\hat{L},\hat{G}}$, it suffices to give the image of $W_\BR$ inside of ${^L}G$.  Recall that $W_\BR$ is an extension of $\BC^\times$ by a group of order $2$, which we write as $\BC^\times \sqcup \sigma\BC^\times$ with $\sigma^2 = -1$.  We give Arthur's construction from Section 5 of \cite{AUC}. The construction of $A_\fq$ in \cite{VZ} depends on an element $\alpha$ of the Lie algebra $\ft$ of a compact torus. Let $T$ be the torus with Lie algebra $\ft$ and let $ \psi_{\hat{L},\hat{G}}: W_\BR \to {^L}G$ be the map sending $\BC^\times$ into $\hat{T}$ so that for any $\lambda^\vee \in X_*(T)$, we have \[ \lambda^\vee(\psi_{\hat{L},\hat{G}}(z)) = z^{\ip{\rho_Q}{\lambda^\vee}}\bar{z}^{-\ip{\rho_Q}{\lambda^\vee}}\] where $\rho_Q = \rho_{\hat{G}}-\rho_{\hat{L}}$, the difference of half-sums of positive roots.  Let the element~$(1 \rtimes \sigma)$ map to~$n_Q \rtimes \sigma$, where for any group~$G$,~$n_G$ is an element in the derived group of~$\hat{G}$ such that~$\ad n_G$ interchanges the positive and negative roots of~$(\hat{G}, \hat{T})$, and with~$n_Q = n_L^{-1}n_G$. 
Putting this together and denoting the embedding of $\hat{L}$ into $\hat{G}$ by $\iota$, define $ \xi_{\hat{L},\hat{G}}(g,w) = \iota(g)\psi_{\hat{L},\hat{G}}(w).$

Now let $\psi_{0,\hat{L}}: SL_2(\BC) \times W_\BR \to {^L}L$ be the Arthur parameter of the packet containing the trivial representation of $L$. It is trivial on $W_\BR$ and sends $SL_2$ to the principal $SL_2$ of $\hat{L}$. Then the Arthur parameter of $G$ corresponding to the Levi subgroup $\hat{L}$ is the composition \[ \psi_{\hat{L}} := \xi_{\hat{L},\hat{G}} \circ \psi_{0,\hat{L}} : SL_2 \times W_\BR \to {^L}G.   \] 
Adams-Johnson \cite{AJ} and more recently Nair-Prasad \cite{NP20Cohomo} have given a description of the packets attached to the parameters $\psi_{\hat{L}}$.
\begin{proposition}[\cite{AJ}, \S 3.3]
	Let $\hat{L}$ be a Levi subgroup of $\hat{G}$,  dual to a Levi~$L(\fq)$ attached to a~$\theta$-stable parabolic subalgebra $\fq$. The parameter $\psi_{\hat{L}} = \xi_{\hat{G}, \hat{L}} \circ \psi_{0,\hat{L}}$ corresponds to a packet $\Pi_{\psi}$ consisting of the representations $A_{\fq}$ such that $\hat{L}(\fq) = \hat{L}$. 
\end{proposition}

We now translate the descriptions of the packets $\Pi_{\psi_{\hat{L}}}$ given in \cite{AJ} and \cite{AUC} into our parametrization by ordered bipartitions.

\begin{proposition} \label{proposition on cohomological packets}
	Let $P = (N_1,...,N_r)$ be an ordered partition of $N$ and~$\psi_P:=\psi_{\hat{L}_P}$ be the corresponding parameter. Then the packet $\Pi_P := \Pi_{\psi_P}$ consists precisely of the cohomological representations $\pi_{B}$ associated to bipartitions $B$ such that $P_B = P$. 
\end{proposition}

\begin{proof}
	We explained above how~$L_B$ gives rise to~$\psi_{\hat{L}_B} = \psi_{\hat{L}_{P_B}} $; the parameters~$\psi_{\hat{L}_B}$ and~$\psi_{\hat{L}_{B'}}$ are equivalent if they are~$\hat{G}$-conjugate. The isomorphism classes of representations~$\pi_{B}$ correspond to Levi subgroups~$L_B$ containing the fixed torus~$T$, so we need only consider conjugation by~$N_{\hat{G}}(\hat{T})$. This action induces an action of the Weyl group~$W(\hat{T}, \hat{G})$ on $\hat{T}$ and on the root datum $(X_*(T), \Del(T), X_*(\hat{T}), \Del(\hat{T}))$. Note that the action of conjugation by $\hat{T}$ on cohomological Arthur parameters will only modify $\psi_{\hat{L}}$ by scaling the entries of $n_Q$. This has no impact on the parameter since $n_Q$ was only specified up to scalars in the construction of $\psi_{\hat{L}}$. 
	
	Thus to determine which Levi subgroups $L_{B'}$ give rise to the conjugacy class of $\hat{L}_P$, we consider the action of $W(\hat{G}, \hat{T})$ (denoted $W(\fg, \ft)$ in \cite{AJ}) on the set of ordered bipartitions. Following the description of $L_B$ given in~\cite{BC05}, ordered bipartitions are determined ultimately by an element $\alpha \in \ft$. The entries of conjugate elements $w \cdot \alpha$ will have the same values, but these values will be distributed differently among the two pieces of $\ft$ belonging to~$U(a)$ and~$U(b)$. We denote the values appearing in the entries of~$\alpha$ by~$z_i$. The data being preserved by conjugation is the number of entries $a_i+b_i$ which are associated to the same value $z_i$, as well as the ordering of the $z_i$. Transitivity of the Weyl group action then ensures that all the possible $B$ such that $P_B$ = $P$ give rise to $\psi_P$. 
\end{proof}

\subsection{Limit Multiplicity for Cohomological Representations} \label{last section on limit multiplicity}
We now give results on growth of cohomology. We return our usual notation, in which $F$ is global, $\cp$ is a prime of $F$, and the subscript ``$\infty$" denotes the collection of all the archimedean places. Fix the set $S_0$ as in Section \ref{section on level structures} so that it contains all but one archimedean place $v_0$. Let $G$ be the inner form of $U_{E/F}(N)$ such that $G_{v_0} \simeq U(a,b)$ and all the other factors at infinity are compact. Define the groups $K(\cp^n)$ and $\Gamma(\cp^n)$ as in Section \ref{section on level structures}. By Matsushima's formula and the inequality \eqref{adelic multiplicity}, we have \begin{equation*} \label{eq matsushima} h^i(\cp^n) := \dim(H^i(\Gamma(\cp^n), \BC) \leq \sum_{\pi = \textbf{1}^{|S_0|}\ten \pi_{v_0} \ten \pi_f} m(\pi)h^i(\fg_{v_0}, K_{v_0}; \pi_{v_0})\dim \pi_f^{K_f(\cp^n)}. \end{equation*}We can now give our theorem for growth of cohomology.
\begin{theorem} \label{final theorem}
	Let~$\psi_\infty$ be the cohomological parameter of~$G_\infty$ associated to a reordering of~$(2k,1,....,1)$. 
	Let \[ h^{i
	}_{\psi_\infty}(\cp^n) = \sum_{\psi \in \Psi(\psi_{\infty})}\sum_{\pi \in \Pi_{\psi}} m(\pi)h^i(\fg_{v_0}, K_{v_0}; \pi_{v_0})\dim \pi_f^{K_f(\cp^n)}. \] Then \[ h^{i}_{\psi_\infty}(\cp^n) \ll \norm(\cp^n)^{N(N-2k)}. \]
\end{theorem}
\begin{proof} 
	The possible contribution to cohomology of a given representation $\pi_{v_0}$ is bounded, so we need only bound the contribution to $m(\pi_\infty, \cp^n)$ coming from packets attached to parameters specializing to $\psi_{\infty}$, for each~$\pi_{\infty} \simeq \pi_{v_0} \ten {\bf 1}^{[F:\BQ]-1} \in \Pi^0_{\psi_{\infty}}$. The result then follows from Theorem~\ref{main theorem 2}, provided that cohomological parameters satisfy its assumptions.  
	From Theorem \ref{Vogan-Zuckerman Theorem} and the following comment, cohomological representations have the  (regular) infinitesimal character of the trivial representation, and trivial central character. From the discussion in Section~\ref{Arthur Packets of Cohomological Representations}, reorderings of~$(2k,1,...,1)$ correspond to parameters for which $\psi(SL_2) = \nu(2k) \oplus \nu(1)^{N-2k}$. Thus the assumptions are satisfied and the result holds. 
\end{proof} 

Note that the theorem does not in fact bound~$m(\pi_\infty, \cp^n)$ for a general~$\pi_{\infty}  \in \Pi_{\psi_\infty}$. Indeed, since Arthur packets are not disjoint, the representation $\pi_\infty$ could also appear in a different Arthur packet whose growth we do not bound. More specifically, if $\pi_\infty = \pi_B \ten  {\bf 1}^{[F:\BQ]-1} \in \Pi^0_{\psi_{\infty}}$ where~$B$ is an ordered bipartition described in Section \ref{computation of cohomology}, it could be the case that~$B$ is the reduction of an ordered bipartition $B'$, for example if we had \[ B = ((1,1),(1,0),(1,0),(0,1)), \quad B' = ((1,1),(2,0),(0,1)). \] On the other hand, if $B$ is not the reduction of another ordered bipartition, then $\pi_\infty = \pi_B$ does not appear in any other archimedean Arthur packet and the theorem produces upper bounds for $m(\pi_{\infty}, \cp^n)$. We record this discussion below.

\begin{corollary}\label{cor on reps to which the theorem applies}
Let $B = ((a_1,b_1),...,(a_r,b_r))$ be a reduced ordered bipartition such that 
\begin{itemize}
	\item[(i)] The associated partition $P_B$ is a reordering of $(2k,1,...,1)$. 
	\item[(ii)] We have $(a_i,b_i) \neq (a_{i+1},b_{i+1})$ for all $i$. 
\end{itemize}
Then \[m(\pi_B, \cp^n) \ll \norm(\cp^n)^{N(N-2k)}.\]
\end{corollary}
\begin{example}
	Assume that $a<b$ in the signature of $U(a,b)$. A family of partitions satisfying the conditions of Corollary \ref{cor on reps to which the theorem applies} are the suitable reorderings of \[B_j = \begin{cases}
	((0,1),...,(1,0),(0,1),(a-j,b-j-2),(0,1),(1,0),...,(0,1)) & N\text{ even }\\
	((a-j,b-j-1),(0,1),(1,0),...,(0,1),(1,0),(0,1)) &  N\text{ odd,}
	\end{cases}\] where $1 \leq j \leq a-1$ if $N$ is even (resp. $0 \leq j \leq a-1$ is $N$ is odd.) The computations of Section \ref{computation of cohomology} show that their lowest degree of cohomology is \[ i = i(N,a,j) = \begin{cases}
	j(N-j-2) + 2a& N \text{ even}\\ 
	j(N-j-1) + a & N \text{ odd.}
	\end{cases} \] Note that $j = \frac{1}{2}(N-2k-2)$ for $N$ even (resp.  $j = \frac{1}{2}(N-2k-1)$ for $N$ odd) which gives the family alluded to in the introduction. 
\end{example}

Additionally, the smallest $i>0$ for which $h^{i}(\cp^n) \neq 0$ is~$i = a$. When $N$ is odd, one can check that representations as above with~$j = 0$ are the only source of cohomology in degree~$a$.  In this situation, we get bounds on Betti numbers. 
\begin{corollary}
	Keeping the assumptions of Theorem \ref{final theorem}, assume additionally that $N$ is odd and that $a<b$. Then\[h^{a}(\cp^n) \ll \norm(\cp^n)^{N}.\]
\end{corollary}
\begin{proof}
This follows from Theorem \ref{final theorem} with $2k = N-1$ provided that 
\[h^a(\fp^n) = \sum_{\psi_\infty}h^a_{\psi_\infty}(\fp^n),\] where the sum is taken over finitely many parameters $\psi_\infty$ associated to a reordering of $(1,N-1)$. Since there are finitely many Arthur of $U(a,b)$ with cohomological representations, this amounts to showing that representations with cohomology in degree $a$ belong only to packets $\Pi_{\psi_\infty}$ associated to these partitions. Going back to Proposition \ref{prop degree computation} and the following discussion, in particular to \eqref{degreeofcohomology}, we find that the only representations with cohomology in degree $a$ are of the form $\pi_B$ for $B$ a reordering of \[ ((0,1),(a,b-1)). \] These cannot be reduced, nor are they the reduction of other ordered bipartitions so by Proposition \ref{proposition on cohomological packets}, they only belong to packets associated to parameters corresponding ordered partitions are reorderings of $(1,a+b-1)=(1,N-1)$, which was exactly our requirement. 
\end{proof}

\subsection{Comparison With the Sarnak-Xue Conjecture} \label{subsection Sarnak-Xue.}

Finally, we compare our results with the conjecture of Sarnak and Xue \cite{SX91} relating multiplicity growth to decay of matrix coefficients. For an irreducible unitary representation $\pi_\infty$ of a Lie group $G$, Sarnak-Xue define \[p(\pi_\infty) = \inf \{ p \geq 2 \mid \text{ $K$-finite matrix coefficients of $\pi_\infty$ are in $L^p(G)$}\}.\] They then conjecture the following bounds for unitary $\pi_\infty$: \[ m(\pi_\infty,\cp^n) \ll_{\epsilon} Vol(X(\cp^n))^{ \frac{2}{p(\pi_{\infty})}+ \epsilon}. \] 
We will now show that the Sarnak-Xue conjecture holds for the representations for which we have proved upper bounds on multiplicity growth. 
\begin{prop} \label{prop exponent inequality}
	Let $\pi_\infty = \pi_B$ be as in Corollary \ref{cor on reps to which the theorem applies}. Then  \[\frac{2}{p(\pi_B)} \geq \frac{N-2k}{N-1}.\]
\end{prop}
Since $m(\pi_B, \cp^n) \ll \norm(\cp^{n})^{N(N-2k)}$ and the volume of $X(\cp^n)$ grows like $\norm(\cp^n)^{N^2-1}$ we obtain the following. 
\begin{corollary}
	For $\pi_B$ as in Corollary \ref{cor on reps to which the theorem applies}, we have \[ m(\pi_B, \cp^n) \ll  \Vol(X(\cp^n))^{\frac{N(N-2k)}{N^2-1}} \ll \Vol(X(\cp^n))^{\frac{2}{p(\pi_B)}} \] and the Sarnak-Xue conjecture holds. 
\end{corollary}
The remainder of the section sets up and gives a proof of Proposition \ref{prop exponent inequality}. 

\subsubsection{Computation of the Rate of Decay.}
For cohomological representations, we will give bounds on $p(\pi_B)$ from the descriptions of $\pi_B$ as Langlands quotients given in \cite{VZ}. For this section we follow the notation of Knapp \cite[\S 7, \S 8]{Kn01}.  We start by bounding~$p(\pi)$ for~$\pi$ an arbitrary Langlands quotient in terms of the inducing data.  

We recall the setup for the definition of Langlands quotients.  First, fix an Iwasawa decomposition 
\begin{equation} \label{eq iwasawa}
	G=KA_0N_0, \quad g=k_ga_gn_g
\end{equation}
of $G$. Here $K$ a maximal compact subgroup,  $A_0$ a maximal split torus, $N_0$ unipotent. By letting $M_0 = Z_K(A_0)$, this gives rise to a minimal parabolic subgroup $S_0= M_0 A_0 N_0$. Let $S=MAN$ be parabolic and regular with respect to $S_0$, which is to say that $A \subset A_0$ is split, $M$ is the no longer necessarily compact Levi component, and $N \subset N_0$ unipotent.  Denote by~$\fa$ (resp.~$\fa_0$) the Lie algebra of~$A$ (resp.~$A_0$), and by~$\fa_M$ the Lie algebra of the maximal split torus of~$M$, so that~$\fa_0 = \fa \oplus \fa_M$. Let~$\rho_0$ be the half-sum of the positive roots of~$\fa_0$ in~$\fg$. Let~$\sigma$ be a discrete series representation of~$M$ and~$\nu \in \fa^*$ be real-valued and in the open positive Weyl chamber. Denote by $\nu_0$ the extension of $\nu$ to $\fa_0$ by setting it to be zero on $\fa_M$, and let $\rho_0 \in \fa_0^*$ be the half-sum of the positive roots of $\fa_0$ in $\fg$. Denote by~$U(S,\sigma,\nu)$ the corresponding parabolically induced representation, and by~$J(S, \sigma, \nu)$ its Langlands quotient. 

The proof of the upcoming  proposition depends on a collection of results from \cite{Kn01}; before stating it we recall the setup and some nomenclature. When studying the decay of matrix coefficients, one introduces a class of so-called \emph{spherical functions} $\phi^G_\nu$ associated to $\nu \in \fa_0^*$, and defined by \[ \phi^G_\nu(g)  = \int_K e^{-(\nu+\rho_0)H(g^{-1}k)}dk \] where $H(g) = \log(a_g) \in \fa_0$ is defined using the Iwasawa decomposition \eqref{eq iwasawa}. Paraphrasing Knapp, these are ``useful yardsticks" to measure the decay of matrix coefficients. For example, it is known that the $K$-finite matrix coefficients of discrete series are dominated by $\phi_0^G$. As their name indicates, the spherical functions are $K$-invariant. A generalization of this notion is that of a \emph{$\tau$-spherical function} associated to a pair $\tau = (\tau_1,\tau_2)$ of representations of $K$: a $\tau$-spherical function is valued in the space $U_1 \ten U_2^\vee$ and is left $\tau_1$- and right $\tau_2$-equivariant. 

In \cite[ \S VII.8]{Kn01}, Knapp studies representations $\pi$ by producing and studying an asymptotic expansion of the $\tau$-spherical functions associated to the $K$-types of $\pi$. The functions $F_{\lambda-\rho_0}$ associated to $\lambda \in \fa_0^*$ appearing in this asymptotic expansion control rates of decay of the $\tau$-spherical function in various directions along $A_0$. The $\lambda-\rho_0$ such that $F_{\lambda-\rho_0}$ contributes nontrivially to the decomposition of the $\tau$-spherical functions of $\pi$ are called exponents of $\pi$. A \emph{leading exponent} of $\pi$ is an exponent $\mu -\rho_0$ of $\pi$, maximal in the sense that for any comparable exponent $\lambda-\rho_0$, the difference $\mu-\lambda$ is a linear combination of simple roots with nonnegative integer coefficients. 

\begin{proposition}\label{lemma from Knapp on bounds}
	Let~$\omega_1,...,\omega_{\dim \fa_0}$ denote the basis of~$\fa_0$ dual to the basis of~$\fa^*_0$ consisting of the simple roots. Then \begin{equation} \label{matrix coefficient ineq for lqnglands quotients} p(J(S, \sigma, \nu)) \leq \inf \{p \geq 2 \mid p \langle \nu_0-\rho_0, \omega_j \rangle  < -2 \langle \rho_0, \omega_j \rangle \text{ for all }\omega_j \}. \end{equation}
\end{proposition}
\begin{proof}
	The lemma follows from the combination of results in \cite{Kn01}. From~\cite[8.48]{Kn01}, the~$K$-finite matrix coefficients of~$J(S, \sigma, \nu)$ belong to~$L^p(G)$ if and only if for all~$\omega_i$ and all  leading exponents~$\mu-\rho_0$ of $J(S,\sigma,\nu)$, the following inequality is satisfied: \begin{equation} \label{eqweightineq}p  \langle \text{Re}\mu-\rho_0, \omega_j\rangle < -2  \langle \rho_0, \omega_j \rangle. \end{equation} 
 In~\cite[8.47]{Kn01}, we see that all leading exponents $\mu$ satisfy $ \ip{\text{Re} \mu}{\omega_j} \leq \ip{\nu_0}{\omega_j}  $ provided there exists an integer $q \geq 0$ such that the~$K$-finite matrix coefficients of~$J(S, \sigma, \nu)$ are bounded above on ~$\overline{A_0^+}$ by a multiple of ~$e^{(\nu_0-\rho)( \log a)}(1+ \|a\|)^q$. This upper bound is established for~$U(S, \sigma, \nu)$ and a fortiori for~$J(S,\sigma,\nu)$ by propositions 7.14 and 7.15 of~\cite{Kn01}, together with the fact that as a discrete series, the~$K$-finite matrix coefficients of~$\sigma$ are dominated by a multiple of the spherical function~$\varphi^M_0$.
\end{proof}
We now give an explicit bounds for~$p(\pi_B)$ in terms of~$B$, for a class of representations including those of Proposition \ref{prop exponent inequality}.
\begin{prop}
	Let~$G = U(a,b)$ with~$a+b = N$, and let~$B = ((a_1,b_1),...,(a_r,b_r))$ be a reduced ordered bipartition of~$(a,b)$. Assume that there is a single index $k$ such that $\min\{a_k,b_k\} \neq 0$, and let $N_k = a_k+b_k$. Then \[
	\frac{2}{p(\pi_B)} \geq \frac{N-N_k}{N-1} .\]
\end{prop}
\begin{proof}
	In light of Proposition \ref{lemma from Knapp on bounds}, we will realize~$\pi_B$ as a Langlands quotient~$J(S,\sigma,\nu)$, and show that for the corresponding~$\nu_0$ we have \[  \inf \{p \geq 2 \mid p \langle \nu_0-\rho_0, \omega_j \rangle  < -2 \langle \rho_0, \omega_j \rangle \text{ for all }\omega_j \} = \frac{2(N-1)}{N-N_k}.\]
	To simplify the computations, note that it is equivalent to show that  \begin{equation}
	\label{equality of q(pib)}  \max_{\omega_j} \left\{\frac{\ip{\nu_0}{\omega_j}}{\ip{\rho_0}{\omega_j}}\right\} = \frac{N_k-1}{N-1}. 	\end{equation}
	
	We recall the descriptions of cohomological representations as Langlands quotients is given in \cite[\S 6]{VZ}. Let~$L_B = \prod_i U(a_i,b_i)$ be the Levi subgroup attached to the representation~$\pi_B$ and fix an Iwasawa decomposition~$L = (K \cap L)AN$ . By assumption,~$A$ has rank~$c_k = \min \{a_k,b_k\}$; denote its Lie algebra by~$\fa$, and let~$\nu$ be the half-sum of the roots of~$\fa$ in the Lie algebra~$\fn$ of~$N$. Let~$M$ be the centralizer of~$A$ in~$G$, and fix~$S$ a choice of parabolic subgroup with Levi~$MA$, such that~$\nu$ is in the open positive Weyl chamber. Then by~\cite[Theorem 6.16]{VZ}, there is a discrete series representation~$\sigma$ of~$M$ such that~$\pi_B \simeq J(S,\sigma, \nu)$. 
	
	To conclude, we put ourselves back in the framework of Proposition \ref{lemma from Knapp on bounds}. Let~$S_0 = A_0 M_0 N_0 \subset S$ be a minimal parabolic subgroup. Then if $\fa_M$ is the Lie algebra of a maximal split torus in $M$, we have $\fa_0 = \fa \oplus \fa_M$ and $\dim \fa_0 = c := \min\{a,b\}$. Let~$\alpha_1,...,\alpha_c$ be the simple roots of~$\fa_0$ in~$\fg$. Recall that~$\nu_0$ is obtained by extending~$\nu$ by~$0$ to~$\fa_M$. Thus we can write~$\nu_0 = \sum_{j=1}^{c_k} j(N_k-j) \alpha_j$ and~$\rho_0 = \sum_{j=1}^c j(N-j) \alpha_j$.  Since the~$\omega_i$ are by construction the dual basis to the~$\alpha_i$, we have \[ \frac{\ip{\nu_0}{\omega_j}}{\ip{\rho_0}{\omega_j}} = \begin{cases}
	\frac{N_k-j}{N-j} & j\leq c_k \\ 0 & j>c_k.
	\end{cases} \] 
	The maximum is achieved when $j=1$.
\end{proof}

\bibliography{refmathilde2}{}

\end{document}